\documentclass[11pt,a4paper]{article}
\usepackage{latexsym,amsfonts,amssymb,amsmath,amsthm}
\usepackage{bbm} 
\usepackage{comment}

\usepackage[usenames,dvipsnames]{color}
\usepackage{ulem}
\usepackage{graphicx}
\usepackage{subfigure}       
\usepackage[section]{placeins} 
\theoremstyle{plain}
\newtheorem{theorem}{Theorem}[section]

\newtheorem{proposition}{Proposition}[section]

\newtheorem{definition}{Definition}[section]

\theoremstyle{remark}
\newtheorem{remark}{Remark}[section]

\numberwithin{equation}{section} \numberwithin{definition}{section}

\def\be{\begin{equation}}
\def\ee{\end{equation}}

\newcommand{\bfa}[1]{\mbox{\boldmath $ #1 $}}

\begin{document}

\title{Stability analysis and simulations of coupled bulk-surface reaction-diffusion systems}
\author{Anotida Madzvamuse,\thanks{Corresponding author: a.madzvamuse@sussex.ac.uk} \footnote{Address for all authors: University of Sussex, School of Mathematical and Physical Sciences, Department of Mathematics, University of Sussex, Pev III, Brighton, BN19QH, UK} \quad
Andy H.W. Chung  \quad and \quad
Chandrasekhar Venkataraman}

\date{}
\maketitle

\noindent{\bf Abstract.}

\bigskip
In this article we formulate new models for coupled systems of bulk-surface reaction-diffusion equations on stationary volumes. The bulk reaction-diffusion equations are coupled to the surface reaction-diffusion equations through linear Robin-type boundary conditions. We then state and prove the necessary conditions for diffusion-driven instability for the coupled system. Due to the nature of the coupling between bulk and surface dynamics, we are able to decouple the stability analysis of the bulk and surface dynamics. Under a suitable choice of model parameter values, the bulk reaction-diffusion system can induce patterning on the surface independent of whether the surface reaction-diffusion system produces or not, patterning. On the other hand, the surface reaction-diffusion system can not generate patterns everywhere in the bulk in the absence of patterning from the bulk reaction-diffusion system. For this case, patterns can only be induced in regions close to the surface membrane. Various numerical experiments are presented to support our theoretical findings. Our most revealing numerical result is that, Robin-type boundary conditions seem to introduce a boundary layer coupling the bulk and surface dynamics.  
\bigskip

\noindent{\bf Key words.}

\bigskip
\noindent Bulk-surface reaction-diffusion equations, bulk-surface finite elements, Turing {\it diffusively-driven} instability, linear stability, pattern formation, Robin-type boundary conditions
\bigskip

\noindent {\bf Mathematics subject classification}

\bigskip
35K55, 35K57, 37B25, 37B55, 37C60, 37C75,

\section{Introduction}
\label{sec:introd}
In many fluid dynamics applications and biological processes, coupled bulk-surface partial differential equations naturally arise in ($2D + 3D$).  In most of these applications and processes, morphological instabilities occur through symmetry breaking resulting in the formation of heterogeneous distributions of chemical substances \cite{levine2005}. In developmental biology, it is essential the emergence and maintenance of polarised states in the form of heterogeneous distributions of chemical substances such as proteins and lipids. Examples of such processes include (but are not limited to) the formation of buds in yeast cells and cell polarisation in biological cells due to responses to external signals through the outer cell membrane  \cite{ratz2012,ratz2013}.  In the context of reaction-diffusion processes, such symmetry breaking arises when a uniform steady state, stable in the absence of diffusion, is driven unstable when diffusion is present thereby giving rise to the formation of spatially inhomogeneous solutions in a process now well-known as the Turing diffusion-driven instability \cite{turing1952}. Classical Turing theory requires that one of the chemical species, typically the {\it inhibitor}, diffuses much faster than the other, the {\it activator} resulting in what is known as the {\it long-range inhibition} and {\it short-range activation} \cite{gierer1972,murray2003}.  

Recently, there has been a surge in studies on models that coupled bulk dynamics to surface dynamics. For example, R\"{a}tz and R\"{o}ger \cite{ratz2013} study symmetry breaking in a bulk-surface reaction-diffusion model for signalling networks. In this work, a single diffusion partial differential equation (the heat equation) is formulated inside the bulk of a cell, while on the cell-surface, a system of two membrane reaction-diffusion equations  
is formulated. The bulk and cell-surface membrane are coupled through Robin-type boundary conditions and a flux term for the membrane system \cite{ratz2013}. Elliott and Ranner \cite{elliott2012} study a finite element approach to a sample elliptic problem: a single elliptic partial differential equation is posed in the bulk and another is posed on the surface. These are then coupled through Robin-type boundary conditions. Novak et al. \cite{novak2007} present an algorithm for solving a diffusion equation on a curved surface coupled to a diffusion model in the volume.  Checkkin et al. \cite{chechkin2012} study bulk-mediated diffusion on planar surfaces. Again, diffusion models are posed in the bulk and on the surface coupling them through boundary conditions. In the area of tissue engineering and regenerative medicine, electrospun membrane are useful in applications such as filtration systems and sensors for chemical detection. Understanding of the fibres' surface, bulk and architectural properties is crucial to the successful development of integrative technology. Nisbet et al. \cite{nisbet2009} presents a detailed review on surface and bulk characterisation of electrospun membranes of porous and fibrous polymer materials.  To explain the long-range proton translocation along biological mombranes, Medvedev and Stuchebrukhov \cite{medvedev2013} propose a model that takes into account the coupled bulk-diffusion that accompanies the migration of protons on the surface. More recently, Rozada {\it et al.,} \cite{rozada2014} presented singular perturbation theory for the stability of localised spot patterns for the Brusselator model on the sphere.

In most of the work above, either elliptic or diffusion models in the bulk have been coupled to surface-elliptic or surface-diffusion or surface-reaction-diffusion models posed on the surface through Robin-type boundary conditions \cite{chechkin2012,elliott2013,medvedev2013,nisbet2009,novak2007,ratz2012,ratz2013}. Here, our focus is to couple systems of reaction-diffusion equations posed both in the bulk and on the surface, setting a mathematical and computational framework to study more complex interactions such as those observed in cell biology, tissue engineering and regenerative medicine, developmental biology  and biopharmaceuticals \cite{chechkin2012,elliott2013,medvedev2013,nisbet2009,novak2007,ratz2012,ratz2013,venkpre}.  We employ the bulk-surface finite element method as introduced by Elliott and Ranner in \cite{elliott2012} to numerically solve the coupled system of  bulk-surface reaction-diffusion equations. Details of the surface-finite element can be found in \cite{dziuk2007}. The bulk  and surface reaction-diffusion systems are coupled through Robin-type boundary conditions. The coupled bulk-surface finite element algorithm is implemented in {\bf deall II} \cite{bangerth2013}.

The key contributions of our work to the theory of pattern formation are:
\begin{itemize}
\item We derive and prove Turing diffusion-driven instability conditions for a coupled system of bulk-surface reaction-diffusion equations. 
\item Using a bulk-surface finite element method,  we approximate the solution to the model system within the bulk and on the boundary surface of a sphere of radius one. 
\item Our results show that if the surface-reaction-diffusion system has the {\it long-range inhibition, short-range activation} form and the bulk-reaction-diffusion system has equal diffusion coefficients, then the surface-reaction-diffusion system can induce patterns in the bulk close to the surface and no patterns form in the interior, far away from the surface. 
\item On the other hand, if the bulk-reaction-diffusion system has the {\it long-range inhibition, short-range activation} form and the surface-reaction-diffusion system has equal diffusion coefficients, then the bulk-reaction-diffusion system can induce pattern formation on the surface. 
\item Furthermore, we prove that  if the bulk and surface reaction-diffusion systems have equal diffusion coefficients, no patterns form. 
\item These theoretical predictions are supported by numerical simulations. 
\end{itemize}

Hence this article is outlined as follows. In Section \ref{sec:prelim} we present the coupled bulk-surface reaction-diffusion system on stationary volumes with appropriate boundary conditions coupling the bulk and surface partial differential equations. The main results of this article are presented in Section \ref{sec:linearstability} where we derive Turing diffusion-driven instability conditions for the coupled system of bulk-surface reaction-diffusion equations. To validate our theoretical findings, we present  bulk-surface finite element numerical solutions in Section \ref{sec:numerics}. In Section \ref{sec:conclusion}, we conclude and discuss the implications of our findings.

\section{Coupled bulk-surface reaction-diffusion systems on stationary volumes}
\label{sec:prelim}
In this section we present a coupled system of bulk-surface reaction-diffusion equations (BSRDEs) posed in a three-dimensional volume as well as on the boundary surface enclosing the volume. We impose Robin-type boundary conditions on the bulk reaction-diffusion system while no boundary conditions are imposed on the surface reaction-diffusion system since the surface is closed.

\subsection{A coupled system of bulk-surface reaction-diffusion equations (BSRDEs) }\label{b_rdes}

Let $\Omega$ be a stationary volume (whose interior is denoted the bulk) enclosed by a compact hypersurface $\Gamma:=\partial \Omega$ which is $C^2$. Also, let $I\!=\![0,T]\; (T>0)$  be some time interval. Moreover, let $\bfa{\nu}$ denote the unit outer normal to $\Gamma$, and let $U$ be any open subset of $\mathbb{R}^{N+1}$ containing $\Gamma$, then  for any function $u$ which is differentiable in $U$, we define the tangential gradient on $\Gamma$ by,
$\nabla_{\Gamma} u = \nabla u - \left(\nabla u \cdot \bfa{\nu}\right) \bfa{\nu},$
where $\cdot$ denotes the regular dot product and $\nabla$ denotes the regular gradient in $\mathbb{R}^{N+1}$.  The tangential gradient is the projection of the regular gradient onto the tangent plane, thus $\nabla_{\Gamma} u \cdot \bfa{\nu} =0$. The Laplace-Beltrami operator on the surface $\Gamma$ is then defined to be the tangential divergence of the tangential gradient
$\Delta_{\Gamma} u = \nabla_{\Gamma} \cdot \nabla_{\Gamma} u$.
For a  vector function $\bfa{u} = (u_1,u_2,\ldots,u_{N+1})\in \mathbb{R}^{N+1}$ the tangential divergence is defined by
\[ \nabla_{\Gamma} \cdot \bfa{u} = \nabla \cdot \bfa{u} - \displaystyle \sum_{i=1}^{N+1}\Big(\nabla u_i \cdot \bfa{\nu} \Big)\nu_i. \]

To proceed, we denote by  $u :\Omega\times I\to\mathbb{R}$ and $v:\Omega\times I\to\mathbb{R}$ two chemical concentrations (species) that react and diffuse in $\Omega$ and  and $r:\Gamma\times I\to\mathbb{R}$ and $s:\Gamma\times I\to\mathbb{R}$  be two chemical species residing only on the surface $\Gamma$ which react and diffuse on the surface. In the absence of cross-diffusion and assuming that coupling is only through the reaction kinetics, we propose to study the following non-dimensionalised coupled system of BSRDEs
\begin{equation}\label{bs_rdes_model0}
\begin{cases}
\begin{cases}
u_t  = \nabla^2 u + \gamma_{ \Omega}  f (u,v,),  \\
v_t  = d_{ \Omega} \nabla^2 v + \gamma_{ \Omega} g (u,v),
\end{cases} \quad \text{in} \; \Omega\times(0,T], \\
\\
\begin{cases}
r_t  =  \nabla_{\Gamma}^2 r + \gamma_{ \Gamma}  \Big( f (r,s) - h_1 (u,v,r,s) \Big),  \\
s_t  = d_{ \Gamma} \nabla_{\Gamma}^2 s + \gamma_{ \Gamma}\Big( g(r,s) - h_2 (u,v,r,s)\Big),
\end{cases}\quad  \quad \text{on} \; \Gamma\times(0,T], \\\end{cases}
\end{equation}
with coupling boundary conditions
\begin{equation}\label{boundary_conditions}
\begin{cases}
 \frac{\partial u }{\partial {\bfa  \nu}}  & =  \gamma_{ \Gamma} h_1 (u,v,r,s), \\
d_{ \Omega} \frac{\partial v }{\partial {\bfa  \nu}}  & = \gamma_{ \Gamma} h_2 (u,v,r,s),
\end{cases} \quad \text{on} \; \Gamma \times(0,T]. \\
\end{equation}
In the above, $\nabla^2 = \frac{\partial ^2 }{\partial x^2} + \frac{\partial ^2 }{\partial y^2} + \frac{\partial ^2 }{\partial z^2}$ represents the Laplacian operator. $d_{ \Omega}$ and $d_{ \Gamma}$ are a positive diffusion coefficients in the bulk and on the surface respectively, representing the ratio between $u$ and $v$, and $r$ and $s$, respectively. $\gamma_{ \Omega}$ and $\gamma_{ \Gamma}$ represent the length scale parameters in the bulk and on the surface respectively. In this formulation, we assume that $f(\cdot,\cdot) $ and $g(\cdot,\cdot) $ are nonlinear reaction kinetics in the bulk and on the surface. $h_1 (u,v,r,s)$ and $h_2 (u,v,r,s)$ are reactions representing the coupling of the internal dynamics in the bulk $\Omega$ to the  surface dynamics on the surface $\Gamma$. As a first attempt, we will consider a more generalised form of  linear coupling of the following nature \cite{macdonald2013}
\begin{align}
h_1 (u,v,r,s) & =  \alpha_1 r - \beta_1 u - \kappa_1 v, \label{hu_eqn1}\\
h_2 (u,v,r,s) & =  \alpha_2 s - \beta_2 u - \kappa_2 v, \label{hv_eqn2}
\end{align}
where $\alpha_1$, $\alpha_2$, $\beta_1$, $\beta_2$, $\kappa_1$ and $\kappa_2$ are constant non-dimensionalised parameters. Initial conditions are given by the positive bounded functions $u_0 ({\bfa  x})$,  $v_0 ({\bfa  x})$, $r_0 ({\bfa  x})$ and $s_0 ({\bfa  x})$.

\subsubsection{Activator-depleted reaction kinetics: An illustrative example}
From now onwards, we restrict our analysis and simulations to the well-known {\it activator-depleted} substrate reaction model \cite{gierer1972,lakkissinum,prigo68,schn79,venkjmb}  also known as
the Brusselator given by
\begin{align}
f(u,v)  = a - u + u^2\,v,\quad \text{and} \quad 
g(u,v)  =b - u^2\,v,\label{schnakv}
\end{align}
where  $a$ and $b$ are positive parameters. For analytical simplicity, we postulate the model system \eqref{bs_rdes_model0} in a more compact form given by
\begin{equation}\label{bs_rdes_model}
\begin{cases}
\begin{cases}
u_t  = \nabla^2 u + f_1 (u,v,r,s),  \\
v_t  = d_{ \Omega} \nabla^2 v + f_2 (u,v,r,s),
\end{cases} \quad {\bfa  x} \; \text{on} \; \Omega,\,\, t>0, \\
\\
\begin{cases}
r_t  =  \nabla_{\Gamma}^2 r +  f_3 (u,v,r,s),  \\
s_t  = d_{ \Gamma} \nabla_{\Gamma}^2 s + f_4(u,v,r,s),
\end{cases}\quad {\bfa  x}\; \text{on} \; \Gamma,\,\, t>0,
\end{cases}
\end{equation}
with coupling boundary conditions \eqref{boundary_conditions}-\eqref{hv_eqn2}. In the above, we have defined appropriately
\begin{align}
f_1 (u,v, r,s) & = \gamma_{ \Omega}  (a - u + u^2 v), \label{f1} \\
f_2 (u,v, r,s) & =  \gamma_{ \Omega}  (b-  u^2 v), \label{f2} \\
f_3 (u,v, r,s) & = \gamma_{ \Gamma}  \big(a - r + r^2 s  - \alpha_1 r + \beta_1 u + \kappa_1 v \big) , \label{f3} \\
f_4 (u,v, r,s) & =  \gamma_{ \Gamma} \big(b -  r^2 s- \alpha_2 s + \beta_2 u + \kappa_2 v \big). \label{f4}
\end{align}

\subsection{Linear stability analysis of the coupled system of BSRDEs}\label{sec:linearstability}

\begin{definition}[Uniform steady state]
A point $(u^*, v^*, r^*, s^*)$ is a uniform steady state of the coupled system of BSRDEs \eqref{bs_rdes_model} with reaction kinetics \eqref{schnakv} if it solves the nonlinear algebraic system given by
$ f_i (u^*,v^*, r^*,s^*) = 0$, for all $ i = 1, \,2, \, 3, \, 4,$
and satisfies  the boundary conditions given by \eqref{boundary_conditions}-\eqref{hv_eqn2}.
\end{definition}

\begin{proposition}[Existence and uniqueness of the uniform steady state]
The coupled system of BSRDEs \eqref{bs_rdes_model} with boundary conditions  \eqref{boundary_conditions} admits a unique steady state given by
\begin{equation}\label{uss}
(u^*,v^*, r^*,s^*) = \left(a+b, \frac{b}{(a+b)^2},a+b, \frac{b}{(a+b)^2}\right),
\end{equation}
provided the following compatibility condition on the coefficients of the coupling is satisfied
\begin{equation} \label{c1}
(\beta_1 - \alpha_1)(\kappa_2 -\alpha_2) - \kappa_1 \beta_2 = 0.
\end{equation}
\end{proposition}

\begin{proof}
The proof follows immediately from the definition of the uniform steady state satisfying reaction kinetics \eqref{f1}-\eqref{f4}. It must be noted that in deriving this unique uniform steady state the compatibility condition  \eqref{c1} coupling bulk and surface dynamics must be satisfied.  
\end{proof}
\begin{remark}
The constraint condition \eqref{c1} on the parameter values $\alpha_i$, $\beta_i$ and $\kappa_i$, $i=1,2$ is a general case of the specific parameter values given in  \cite{macdonald2013} where the following parameter values where selected
$\alpha_1 = \beta_1 = \frac{5}{12}$,
$\alpha_2 = \kappa_2 = 5$,
$\kappa_1 = 0$, and 
$\beta_2 = 0$.

\end{remark}
\begin{remark}
Note that there exists an infinite number of solutions to problem \eqref{c1}.
\end{remark}

\subsubsection{Linear stability analysis in the absence of diffusion}

Next, we study Turing diffusion-driven instability for the coupled system of BSRDEs \eqref{bs_rdes_model0}-\eqref{hv_eqn2} with reaction kinetics \eqref{schnakv}. To proceed, we first consider the linear stability of the spatially uniform steady state. For convenience's sake, let us denote by
${\bfa  w} = \begin{pmatrix} u ,\; v, \; r, \; s \end{pmatrix}^T,$
the vector of the species  $u$, $v$, $r$ and $s$. Furthermore, defining the vector ${\bfa \xi}$ such that $\vert\xi_i\vert <<1$ for all $i=1$, $2$, $3$ and $4$, it follows that writing
${\bfa  w} = {\bfa  w}^* + {\bfa  \xi},$
the linearized system of coupled BSRDEs can be posed as
\begin{equation}\label{linear_ode}
{\bfa  w}_t = {\bfa  \xi}_t = {\bfa  J}_{\bfa  F} {\bfa  \xi},
\end{equation}
where ${\bfa  J}_{\bfa  F}$ represents the Jacobian matrix representing the first linear terms of the linearization process. Its entries are defined by
\begin{align}
{\bfa  J}_{\bfa  F} = \begin{pmatrix} \frac{\partial f_1}{\partial u} &  \frac{\partial f_1}{\partial v} &  \frac{\partial f_1}{\partial r} &  \frac{\partial f_1}{\partial s} \\
\frac{\partial f_2}{\partial u} &  \frac{\partial f_2}{\partial v} &  \frac{\partial f_2}{\partial r} &  \frac{\partial f_2}{\partial s} \\
\frac{\partial f_3}{\partial u} &  \frac{\partial f_3}{\partial v} &  \frac{\partial f_3}{\partial r} &  \frac{\partial f_3}{\partial s} \\
\frac{\partial f_4}{\partial u} &  \frac{\partial f_4}{\partial v} &  \frac{\partial f_4}{\partial r} &  \frac{\partial f_4}{\partial s}
  \end{pmatrix} & = \begin{pmatrix} {f_1}_u& {f_1}_v & 0 & 0 \\
 {f_2}_u & {f_2}_v & 0 & 0 \\
{f_3}_u &  {f_3}_v & {f_3}_r & {f_3}_s \\
 {f_4}_u & {f_4}_v & {f_4}_r & {f_4}_s
  \end{pmatrix} \notag \\
  & := \begin{pmatrix} {f}_u& {f}_v & 0 & 0 \\
 {g}_u & {g}_v & 0 & 0 \\
-{h_1}_u &  -{h_1}_v & {f}_r - {h_1}_r & {f}_s - {h_1}_s \\
-{h_2}_u &  -{h_2}_v & {g}_r - {h_2}_r & {g}_s - {h_2}_s 
  \end{pmatrix}.
\label{JF}
\end{align}
where by definition ${f_1}_u: =\frac{\partial f_1}{\partial u}$ represents a partial derivative of $f_1 (u,v)$ with respect to $u$.  We are looking for solutions to the system of linear ordinary differential equations \eqref{linear_ode} which are of the form
${\bfa  \xi} \propto e^{\lambda t}$.
Substituting into \eqref{linear_ode}, results in the following classical eigenvalue problem
\begin{equation}
\Big| \lambda {\bfa  I} - {\bfa  J}_{\bfa  F}  \Big| = 0,
\end{equation}
where ${\bfa  I}$ is the identity matrix. Making appropriate substitutions and carrying out standard calculations we obtain the following dispersion relation for $\lambda$
\begin{align}
\Big| \lambda {\bfa  I} - {\bfa  J}_{\bfa  F}  \Big| & = \begin{vmatrix} \lambda - {f_1}_u& {f_1}_v & 0 & 0 \\
 {f_2}_u &  \lambda - {f_2}_v & 0 & 0 \\
{f_3}_u &  {f_3}_v &  \lambda - {f_3}_r & {f_3}_s \\
 {f_4}_u & {f_4}_v & {f_4}_r & \lambda -  {f_4}_s
  \end{vmatrix} =0, \notag \\
\iff p_4 (\lambda) &= \lambda^4 + a_1 \lambda^3 + a_2 \lambda^2 + a_3 \lambda  + a_4 = 0, \notag
\end{align}
where
\begin{align}
a_1 &  = - \Big({f_1}_u + {f_2}_v+ {f_3}_r + {f_4}_s \Big), \label{a1} \\
a_2 & = ({f_1}_u {f_2}_v - {f_1}_v {f_2}_u) + ({f_3}_r {f_4}_s - {f_3}_s {f_4}_r)  + ({f_1}_u + {f_2}_v)({f_3}_r + {f_4}_s), \label{a2}\\
a_3 & = -\Big[({f_1}_u {f_2}_v - {f_1}_v {f_2}_u)({f_3}_r + {f_4}_s) +  ({f_3}_r {f_4}_s - {f_3}_s {f_4}_r) ({f_1}_u + {f_2}_v)\Big], \label{a3} \\
a_4 & = ({f_1}_u {f_2}_v - {f_1}_v {f_2}_u)  ({f_3}_r {f_4}_s - {f_3}_s {f_4}_r). \label{a4}
\end{align}
For convenience's sake, let us denote by
\begin{equation}
\left({\bfa  J}_{\bfa  F}\right)_{\Omega} : =  \begin{pmatrix} {f_1}_u& {f_1}_v  \\
 {f_2}_u & {f_2}_v
 \end{pmatrix} \quad 
 \text{and} \quad 
\left({\bfa  J}_{\bfa  F}\right)_{\Gamma} : =  \begin{pmatrix} {f_3}_r & {f_3}_s \\
 {f_4}_r & {f_4}_s \end{pmatrix}
\end{equation}
the submatrices of  ${\bfa  J}_{\bfa  F}$ corresponding to the bulk reaction kinetics and the surface reaction kinetics respectively. We can now define
\begin{align}
& \text{Tr} \left({\bfa  J}_{\bfa  F}\right)  := {f_1}_u + {f_2}_v+ {f_3}_r + {f_4}_s, \qquad
\text{Tr} \left({\bfa  J}_{\bfa  F}\right)_{\Omega}   := {f_1}_u + {f_2}_v, \quad 
\text{Tr} \left({\bfa  J}_{\bfa  F}\right)_{\Gamma}   := {f_3}_r + {f_4}_s, \notag \\
& \text{Det} \left({\bfa  J}_{\bfa  F}\right)_{\Omega}   := {f_1}_u {f_2}_v - {f_1}_v {f_2}_u, \quad \text{and} \quad 
\text{Det} \left({\bfa  J}_{\bfa  F}\right)_{\Gamma}   := {f_3}_r {f_4}_s - {f_3}_s {f_4}_r. \notag
\end{align}

\begin{theorem}[Necessary and sufficient conditions for $\text{Re} (\lambda)<0$]\label{ns_cond_ad}
The necessary and sufficient conditions such that the zeros of the polynomial $p_4 (\lambda)$ have $\text{Re} (\lambda)<0$ are given by the following conditions  
\begin{align}
& \text{Tr} \left({\bfa  J}_{\bfa  F}\right)   <0, \label{ws_cond1} \\
& \text{Det} \left({\bfa  J}_{\bfa  F}\right)_{\Omega} + \text{Det} \left({\bfa  J}_{\bfa  F}\right)_{\Gamma} + \text{Tr} \left({\bfa  J}_{\bfa  F}\right)_{\Omega}\text{Tr} \left({\bfa  J}_{\bfa  F}\right)_{\Gamma}  >0, \ \label{ws_cond2}\\
& \text{Det} \left({\bfa  J}_{\bfa  F}\right)_{\Omega}  \text{Tr} \left({\bfa  J}_{\bfa  F}\right)_{\Gamma} + \text{Det} \left({\bfa  J}_{\bfa  F}\right)_{\Gamma} \text{Tr} \left({\bfa  J}_{\bfa  F}\right)_{\Omega}  <0,  \label{ws_cond3} \\
& \text{Det} \left({\bfa  J}_{\bfa  F}\right)_{\Omega}  \text{Det} \left({\bfa  J}_{\bfa  F}\right)_{\Gamma}  >0,  \label{ws_cond4} \\
& \Big[\text{Tr} \left({\bfa  J}_{\bfa  F}\right)_{\Gamma} \text{Tr} \left({\bfa  J}_{\bfa  F}\right)  - 2  \text{Det} \left({\bfa  J}_{\bfa  F}\right)_{\Omega}  \Big]  \text{Tr} \left({\bfa  J}_{\bfa  F}\right)_{\Omega}  + \Big[\text{Tr} \left({\bfa  J}_{\bfa  F}\right)_{\Omega} \text{Tr} \left({\bfa  J}_{\bfa  F}\right)  \Big.  \notag \\ & \Big. \hspace{20em}- 2  \text{Det} \left({\bfa  J}_{\bfa  F}\right)_{\Gamma}  \Big]  \text{Tr} \left({\bfa  J}_{\bfa  F}\right)_{\Gamma}    >0,   \label{ws_cond5} \\
& \left[ \big(\text{Det} \left({\bfa  J}_{\bfa  F}\right)_{\Omega}  +  \text{Det} \left({\bfa  J}_{\bfa  F}\right)_{\Gamma} \big)^2 - \big( \text{Det} \left({\bfa  J}_{\bfa  F}\right)_{\Omega} \text{Tr} \left({\bfa  J}_{\bfa  F}\right)_{\Gamma}  \right. \notag \\ & \left. \hspace{8em} + \text{Det} \left({\bfa  J}_{\bfa  F}\right)_{\Gamma} \text{Tr} \left({\bfa  J}_{\bfa  F}\right)_{\Omega} \big) \text{Tr} \left({\bfa  J}_{\bfa  F} \right)  \right] \text{Tr} \left({\bfa  J}_{\bfa  F}\right)_{\Omega} \text{Tr} \left({\bfa  J}_{\bfa  F}\right)_{\Gamma}>0.  \label{ws_cond6}
\end{align}
\end{theorem}

 \begin{proof}
The proof enforces that $p_4 (\lambda)$ is a Hurwitz polynomial and therefore satisfies the Routh-Hurwitz criterion in order for $\text{Re} (\lambda)<0$. The first condition to be satisfied is that  $a_4 \ne 0$ otherwise $\lambda =0$ is a trivial root, thereby reducing the 4-th order polynomial to a cubic polynomial.
The first four conditions are a result of requiring that each coefficient $a_i$ with $i=1$, $2$, $3$ and $4$ of the polynomial $p_4 (\lambda)$ are all positive. The rest of the conditions are derived as shown below.

We require that the determinant of the matrix
 \[ \begin{vmatrix} a_1 & a_3 \\
 1 & a_2 \end{vmatrix} = a_1 a_2 - a_3 >0.
\]
Substituting $a_1$, $a_2$ and $a_3$ appropriately we obtain
\begin{align}
 \Big[ \text{Det} \left({\bfa  J}_{\bfa  F}\right)_{\Omega}  & + \text{Det} \left({\bfa  J}_{\bfa  F}\right)_{\Gamma}   + \text{Tr} \left({\bfa  J}_{\bfa  F}\right)_{\Omega} \text{Tr} \left({\bfa  J}_{\bfa  F}\right)_{\Gamma}\Big]  \Big[ - \text{Tr} \left({\bfa  J}_{\bfa  F}\right) \Big]   \notag \\
&+ \Big[ \text{Det} \left({\bfa  J}_{\bfa  F}\right)_{\Omega} \text{Tr} \left({\bfa  J}_{\bfa  F}\right)_{\Gamma} + \text{Det} \left({\bfa  J}_{\bfa  F}\right)_{\Gamma} \text{Tr} \left({\bfa  J}_{\bfa  F}\right)_{\Omega}  \Big] > 0.
\end{align}
Exploiting the fact that
\[ \text{Tr} \left({\bfa  J}_{\bfa  F}\right) = \text{Tr} \left({\bfa  J}_{\bfa  F}\right)_{\Omega}+\text{Tr} \left({\bfa  J}_{\bfa  F}\right)_{\Gamma},\]
it then follows that
\begin{align}
&  a_1 a_2 - a_3  \notag \\
& = \text{Tr} \left({\bfa  J}_{\bfa  F}\right)_{\Omega} \text{Tr} \left({\bfa  J}_{\bfa  F}\right)_{\Gamma} \text{Tr} \left({\bfa  J}_{\bfa  F}\right) - \Big[ \text{Det} \left({\bfa  J}_{\bfa  F}\right)_{\Omega} \text{Tr} \left({\bfa  J}_{\bfa  F}\right)_{\Omega} + \text{Det} \left({\bfa  J}_{\bfa  F}\right)_{\Gamma} \text{Tr} \left({\bfa  J}_{\bfa  F}\right)_{\Gamma}\Big] >0 \notag
 \end{align}
if and only if
 \begin{align}
 \frac{1}{2} \Big[\text{Tr} \left({\bfa  J}_{\bfa  F}\right)_{\Gamma} \text{Tr} \left({\bfa  J}_{\bfa  F}\right)  &-  2  \text{Det} \left({\bfa  J}_{\bfa  F}\right)_{\Omega}\Big] \text{Tr} \left({\bfa  J}_{\bfa  F}\right)_{\Omega}  \notag \\
& + \frac{1}{2} \Big[\text{Tr} \left({\bfa  J}_{\bfa  F}\right)_{\Omega} \text{Tr} \left({\bfa  J}_{\bfa  F}\right) -  2  \text{Det} \left({\bfa  J}_{\bfa  F}\right)_{\Gamma}\Big] \text{Tr} \left({\bfa  J}_{\bfa  F}\right)_{\Gamma}>0.   \notag
\end{align}
Multiplying throughout by 2 we obtain condition \eqref{ws_cond5} in Theorem \ref{ns_cond_ad}.

The last condition results from imposing the condition that
\[ \begin{vmatrix} a_1 & a_3 & 0 \\
 1 & a_2 & a_4 \\
 0 & a_1 & a_3  \end{vmatrix} = a_3 (a_1 a_2 - a_3) - a_1^2 a_4  >0.
\]
It can be shown that
\begin{align}
 & a_3 (a_1 a_2 - a_3)  \notag \\
& = - \Big[ \text{Det} \left({\bfa  J}_{\bfa  F}\right)_{\Omega} \text{Tr} \left({\bfa  J}_{\bfa  F}\right)_{\Omega}  \text{Tr}^2 \left({\bfa  J}_{\bfa  F}\right)_{\Gamma}  +  \text{Det} \left({\bfa  J}_{\bfa  F}\right)_{\Gamma} \text{Tr}^2 \left({\bfa  J}_{\bfa  F}\right)_{\Omega} \text{Tr} \left({\bfa  J}_{\bfa  F}\right)_{\Gamma} \Big] \text{Tr} \left({\bfa  J}_{\bfa  F}\right)  \notag \\
&  + \text{Det}^2 \left({\bfa  J}_{\bfa  F}\right)_{\Omega} \text{Tr} \left({\bfa  J}_{\bfa  F}\right)_{\Omega}  \text{Tr} \left({\bfa  J}_{\bfa  F}\right)_{\Gamma}  + \text{Det} \left({\bfa  J}_{\bfa  F}\right)_{\Omega} \text{Det} \left({\bfa  J}_{\bfa  F}\right)_{\Gamma} \text{Tr}^2 \left({\bfa  J}_{\bfa  F}\right)_{\Omega}   \notag \\
 &  + \text{Det} \left({\bfa  J}_{\bfa  F}\right)_{\Omega} \text{Det} \left({\bfa  J}_{\bfa  F}\right)_{\Gamma}  \text{Tr}^2 \left({\bfa  J}_{\bfa  F}\right)_{\Omega}  + \text{Det}^2 \left({\bfa  J}_{\bfa  F}\right)_{\Gamma}  \text{Tr} \left({\bfa  J}_{\bfa  F}\right)_{\Omega}\text{Tr} \left({\bfa  J}_{\bfa  F}\right)_{\Gamma}. \label{part1}
\end{align}
On the other hand,
\begin{align}
& a_1^2 a_4  = \text{Tr}^2 \left({\bfa  J}_{\bfa  F}\right) \text{Det} \left({\bfa  J}_{\bfa  F}\right)_{\Omega}  \text{Det} \left({\bfa  J}_{\bfa  F}\right)_{\Gamma} \notag \\
& = \Big( \text{Tr} \left({\bfa  J}_{\bfa  F}\right)_{\Omega}+\text{Tr} \left({\bfa  J}_{\bfa  F}\right)_{\Gamma}\Big)^2 \text{Det} \left({\bfa  J}_{\bfa  F}\right)_{\Omega}  \text{Det} \left({\bfa  J}_{\bfa  F}\right)_{\Gamma}   \notag \\
&  = \text{Det} \left({\bfa  J}_{\bfa  F}\right)_{\Omega}  \text{Det} \left({\bfa  J}_{\bfa  F}\right)_{\Gamma} \text{Tr}^2 \left({\bfa  J}_{\bfa  F}\right)_{\Omega}  + 2 \text{Det} \left({\bfa  J}_{\bfa  F}\right)_{\Omega}  \text{Det} \left({\bfa  J}_{\bfa  F}\right)_{\Gamma} \text{Tr} \left({\bfa  J}_{\bfa  F}\right)_{\Omega} \text{Tr} \left({\bfa  J}_{\bfa  F}\right)_{\Gamma}  \notag  \\
& \hspace{3cm} + \text{Det} \left({\bfa  J}_{\bfa  F}\right)_{\Omega}  \text{Det} \left({\bfa  J}_{\bfa  F}\right)_{\Gamma}  \text{Tr}^2 \left({\bfa  J}_{\bfa  F}\right)_{\Gamma}.  \label{part2}
\end{align}
Hence combining \eqref{part1} and \eqref{part2} and simplifying conveniently we have
\begin{equation}
\begin{split}
a_3 (a_1 a_2 - a_3) - a_1^2 a_4 = \left[ \big(\text{Det} \left({\bfa  J}_{\bfa  F}\right)_{\Omega}  +  \text{Det} \left({\bfa  J}_{\bfa  F}\right)_{\Gamma} \big)^2 - \big( \text{Det} \left({\bfa  J}_{\bfa  F}\right)_{\Omega} \text{Tr} \left({\bfa  J}_{\bfa  F}\right)_{\Gamma} \right.  \notag \\
 \left. + \text{Det} \left({\bfa  J}_{\bfa  F}\right)_{\Gamma} \text{Tr} \left({\bfa  J}_{\bfa  F}\right)_{\Omega} \big) \text{Tr} \left({\bfa  J}_{\bfa  F} \right)  \right] \times \text{Tr} \left({\bfa  J}_{\bfa  F}\right)_{\Omega} \text{Tr} \left({\bfa  J}_{\bfa  F}\right)_{\Gamma} >0,
\end{split}
\end{equation}
 resulting in condition \eqref{ws_cond6}. 
 \end{proof}

\begin{remark}
The characteristic polynomial 
\[ p_4 (\lambda) = \lambda^4 + a_1 \lambda^3 + a_2 \lambda^2 + a_3 \lambda  + a_4\]
can also be written more compactly in the form of 
 \[p_4 (\lambda) = \Big( \lambda^2 + \lambda ({f_1}_u + {f_2}_v)  + {f_1}_u {f_2}_v - {f_1}_v {f_2}_u  \Big) \Big(\lambda^2 + \lambda ({f_3}_r + {f_4}_s)  + {f_3}_r {f_4}_s - {f_3}_s {f_4}_r  \Big) \]
 thereby coupling the bulk and surface dispersion relations in the absence of spatial variations.
\end{remark}

\subsubsection{Linear stability analysis in the presence of diffusion}

 Next we introduce spatial variations and study under what conditions the uniform steady state is linearly unstable. We linearize around the uniform steady state by taking small spatially varying perturbations of the form
\begin{equation}\label{pertb}
{\bfa  w} ({\bfa  x},t) = {\bfa  w}^* +\epsilon {\bfa  \xi} ({\bfa  x},t), \quad \text{with} \quad \epsilon << 1.
\end{equation}
Substituting \eqref{pertb} into the coupled system of  BSRDEs   \eqref{bs_rdes_model0}-\eqref{hv_eqn2} with reaction kinetics \eqref{schnakv} we obtain a linearized system of partial differential equations
\begin{align}
{\xi_1}_t & = \nabla^2 \xi_1 + \gamma_{ \Omega} \Big(f_u \xi_1 + f_v \xi_2 \Big), \label{x1t} \\
{\xi_2}_t & =d_{ \Omega} \nabla^2 \xi_2 + \gamma_{ \Omega} \Big(g_u \xi_1 + g_v \xi_2 \Big), \label{x2t} \\
{\xi_3}_t & = \nabla_{ \Gamma}^2 \xi_3 + \gamma_{ \Gamma} \Big(f_r \xi_3 + f_s \xi_4 - {h_1}_u \xi_1 - {h_1}_v \xi_2- {h_1}_r \xi_3- {h_1}_s \xi_4 \Big), \label{x3t} \\
{\xi_4}_t & = d_{ \Gamma} \nabla_{ \Gamma}^2 \xi_4 + \gamma_{ \Gamma} \Big(g_r \xi_3 + g_s \xi_4 - {h_2}_u \xi_1 - {h_2}_v \xi_2- {h_2}_r \xi_3- {h_2}_s \xi_4 \Big), \label{x4t}
\end{align}
with linearised boundary conditions
\begin{align}
\frac{\partial \xi_1}{\partial \nu} & =  \gamma_{ \Gamma} \Big({h_1}_u \xi_1 + {h_1}_v \xi_2+ {h_1}_r \xi_3+ {h_1}_s \xi_4 \Big), \label{bcx1n} \\
d_{ \Gamma} \frac{\partial \xi_2}{\partial \nu} & =  \gamma_{ \Gamma} \Big({h_2}_u \xi_1 + {h_2}_v \xi_2+ {h_2}_r \xi_3+ {h_2}_s \xi_4 \Big). \label{bcx2n}
\end{align}
In the above, we have used the original reaction kinetics for the purpose of clarity.

In order to proceed, we restrict our analysis to circular and spherical domains where we can transform the cartesian coordinates into polar coordinates and be able to exploit the method of separation of variables. Without loss of generality, we write the following eigenvalue problem in the bulk
\begin{equation}
\nabla^2 \psi_{k_{l,m}} (r) = - k_{l,m}^2 \psi_{k_{l,m}} (r), \quad 0 < r < 1,
\end{equation}
where each $\psi_k$ satisfies the boundary conditions \eqref{bcx1n} and \eqref{bcx2n}.
On the surface the eigenvalue problem is posed as
\begin{equation}
\nabla_{ \Gamma}^2 \phi (y) = - l (l+1) \phi (y), \quad y \in \Gamma.
\end{equation}
\begin{remark}
For the case of circular and spherical domains, if $r=1$, then $k_{l,m}^2 = l(l+1)$.
\end{remark}

Taking $x\in \mathbb{B}$,  $y \in \Gamma$, then writing in polar coordinates $x=ry$, $r\in (0,1)$ we can define, for all $l \in \mathbb{N}_0$,  $m \in \mathbb{Z}$, $|m| \le l$, the following power series solutions \cite{ratz2012,ratz2013}
\begin{align}
\xi_1 (ry, t) & = \sum u_{l,m} (t) \psi_{k_{l,m}} (r) \phi_{l,m} (y), \quad
\xi_2 (ry, t)  = \sum v_{l,m} (t) \psi_{k_{l,m}} (r) \phi_{l,m} (y), \label{xry2} \\
\xi_3 (y, t) & = \sum r_{l,m} (t)  \phi_{l,m} (y), \quad \text{and} \quad 
\xi_4 (y, t)  = \sum s_{l,m} (t)  \phi_{l,m} (y). \label{xry4}
\end{align}
On the  surface, substituting the power series solutions \eqref{xry4} into \eqref{x3t} and \eqref{x4t} we have
\begin{align}
\frac{d r_{l,m}}{dt} & = - l(l+1) r_{l,m} + \gamma_{ \Gamma} \Big(f_r r_{l,m} + f_s s_{l,m} \Big) \notag \\
& -\gamma_{ \Gamma}\Big({h_1}_u u_{l,m} \psi_{k_{l,m}} (1)  + {h_1}_v v_{l,m} \psi_{k_{l,m}} (1) + {h_1}_r r_{l,m}+ {h_1}_s s_{l,m} \Big), \label{dx3t} \\
\frac{d s_{l,m}}{dt} & = - d_{ \Gamma} l(l+1)  s_{l,m} + \gamma_{ \Gamma} \Big(g_r r_{l,m} + g_s s_{l,m} \Big) \notag \\
& -\gamma_{ \Gamma}\Big({h_2}_u u_{l,m} \psi_{k_{l,m}} (1)  + {h_2}_v v_{l,m} \psi_{k_{l,m}} (1) + {h_2}_r r_{l,m}+ {h_2}_s s_{l,m} \Big). \label{dx4t}
\end{align}
Similarly, substituting the power series solutions \eqref{xry2} into the bulk equations \eqref{x1t} and \eqref{x2t} we obtain the following system of ordinary differential equations
\begin{align}
\frac{d u_{l,m}}{dt} & = - k_{l,m}^2 u_{l,m} + \gamma_{ \Omega} \Big(f_u u_{l,m} + f_v v_{l,m} \Big), \label{dx1t} \\
\frac{d v_{l,m}}{dt} & = - d_{ \Omega} k_{l,m}^2 v_{l,m} + \gamma_{ \Omega} \Big(g_u u_{l,m} + g_v v_{l,m} \Big). \label{dx2t}
\end{align}
Equations \eqref{dx1t} and \eqref{dx2t} are supplemented with boundary conditions
\begin{align}
u_{l,m} \psi'_{k_{l,m}} (1) & = \gamma_{ \Gamma}\Big({h_1}_u u_{l,m} \psi_{k_{l,m}} (1)  + {h_1}_v v_{l,m} \psi_{k_{l,m}} (1) + {h_1}_r r_{l,m}+ {h_1}_s s_{l,m} \Big), \label{bcdx3t} \\
d_{ \Omega} v_{l,m} \psi'_{k_{l,m}} (1) & =\gamma_{ \Gamma}\Big({h_2}_u u_{l,m} \psi_{k_{l,m}} (1)  + {h_2}_v v_{l,m} \psi_{k_{l,m}} (1) + {h_2}_r r_{l,m}+ {h_2}_s s_{l,m} \Big), \label{bcdx4t}
\end{align}
where ${\psi'_{k_{l,m}}} : = \Big. \frac{d \psi_{k_{l,m}} (r)}{d r} \Big|_{r=1}$.
Writing
\[ \begin{pmatrix} u_{l,m}, v_{l,m},  r_{l,m},  s_{l,m},  \end{pmatrix}^T  = \begin{pmatrix} u_{l,m}^0,  v_{l,m}^0, r_{l,m}^0, s_{l,m}^0 \end{pmatrix}^T e^{\lambda_{l,m} t}, \]
and substituting into the system of ordinary differential equations \eqref{dx3t}-\eqref{dx2t}, we obtain the following eigenvalue problem
\begin{equation}
\Big(\lambda_{l,m} {\bfa I} + {\bfa M} \Big) {\bfa \xi}_{l,m}^0 = {\bfa 0}
\end{equation}
where
\begin{equation}
{\bfa M} = \begin{pmatrix}
 k_{l,m}^2 - \gamma_{ \Omega}f_u & -\gamma_{ \Omega}f_v & 0 & 0 \\
 - \gamma_{ \Omega} g_u & d_{ \Omega} k_{l,m}^2 - \gamma_{ \Omega} g_v & 0 & 0 \\
 \psi'_{k_{l,m}} (1) &  0 &  l(l+1) -\gamma_{ \Gamma} f_r &-\gamma_{ \Gamma} f_s \\
 0 & d_{ \Omega} \psi'_{k_{l,m}} (1) & -\gamma_{ \Gamma} g_r & d_{ \Gamma} l(l+1)  -\gamma_{ \Gamma} g_s
\end{pmatrix}, \notag
\end{equation}
and
\[{\bfa \xi}_{l,m}^0 =  \begin{pmatrix} u_{l,m}^0, v_{l,m}^0,  r_{l,m}^0, s_{l,m}^0 \end{pmatrix}^T.\]
Note that the boundary conditions \eqref{bcdx3t} and \eqref{bcdx4t} have been applied appropriately to the surface linearised reaction-diffusion equations. Since
\[ \begin{pmatrix} u_{l,m}^0, v_{l,m}^0, r_{l,m}^0, s_{l,m}^0 \end{pmatrix}^T \ne \begin{pmatrix} 0,  0,  0,  0 \end{pmatrix}^T, \]
it follows that the coefficient matrix must be singular, hence we require that
\[ \Big|\lambda_{l,m} {\bfa I} + {\bfa M} \Big|  = 0.\]
 Straight forward calculations show that the eigenvalue $\lambda_{l,m}$ solves the following dispersion relation written in compact form as
\begin{equation}
\Big(\lambda_{l,m}^2 + \text{Tr} \left({\bfa  M}\right)_{\Omega} \lambda_{l,m} + \text{Det} \left({\bfa  M} \right)_{\Omega}  \Big) \Big( \lambda_{l,m}^2 +  \text{Tr} \left({\bfa  M} \right)_{\Gamma} \lambda_{l,m} + \text{Det} \left({\bfa  M} \right)_{\Gamma}  \Big) = 0,
\end{equation}
where we have defined conveniently
\begin{align}
 \text{Tr} ({\bfa  M})_{\Omega} & :=  (d_{ \Omega} +1)k_{l,m}^2  -\gamma_{ \Omega} (f_u + g_v),  \notag \\
 \text{Tr} ({\bfa  M})_{\Gamma} & := (d_{ \Gamma} +1 )l(l+1) - \gamma_{ \Gamma} (f_r +g_s) , \notag \\
 \text{Det} ({\bfa  M})_{\Omega} & := d_{ \Omega} k_{l,m}^4 - \gamma_{ \Omega}\left( d_{ \Omega} f_u  + g_v \right)  k_{l,m}^2 + \gamma_{ \Omega}^2 (f_u g_v  -f_v g_u), \notag \\
 \text{Det} ({\bfa  M})_{\Gamma} & := d_{ \Gamma} l^2(l+1)^2 - \gamma_{ \Gamma} \left( d_{ \Gamma} f_r  + g_s \right)  l(l+1)+   \gamma_{ \Gamma}^2 (f_r g_s - f_s g_r).  \notag
\end{align}
The above holds true if and only if either
\begin{equation}\label{disp_omega}
\lambda_{l,m}^2 + \text{Tr} \left({\bfa  M}\right)_{\Omega} \lambda_{l,m} + \text{Det} \left({\bfa  M} \right)_{\Omega}   = 0,
\end{equation}
or
\begin{equation}\label{disp_gamma}
\lambda_{l,m}^2 +  \text{Tr} \left({\bfa  M} \right)_{\Gamma} \lambda_{l,m} + \text{Det} \left({\bfa  M} \right)_{\Gamma}  = 0.
\end{equation}
In the presence of diffusion, we require the emergence of  spatial growth. In order for the uniform steady state ${\bfa  w}^*$ to be unstable we require that either
\begin{enumerate}
\item $\text{Re} (\lambda_{l,m} (k_{l,m}^2))>0$ for some $k_{l,m}^2>0$,
\item[] or
 \item $\text{Re} (\lambda_{l,m} (l(l+1)))>0$ for some $l(l+1)>0$,
 \item[] or
 \item both.
\end{enumerate}
Solving \eqref{disp_omega} (and  similarly \eqref{disp_gamma}) we obtain the eigenvalues
\begin{equation}\label{lambda_omega}
2 \text{Re} (\lambda_{l,m} (k_{l,m}^2)) = - \text{Tr} \left({\bfa  M}\right)_{\Omega} \pm \sqrt{ \text{Tr}^2 \left({\bfa  M}\right)_{\Omega} - 4 \text{Det} \left({\bfa  M} \right)_{\Omega}  }.
\end{equation}
It follows then that $\text{Re} (\lambda_{l,m} (k_{l,m}^2))>0$ for some $k_{l,m}^2>0$ if and only if the following conditions hold:
\begin{equation}
\begin{cases}
\text{Tr} \left({\bfa  M}\right)_{\Omega} <0   \iff (d_{ \Omega} +1)k_{l,m}^2  -\gamma_{ \Omega} (f_u + g_v)<0,  \quad \text{and}  \\ \\
 \text{Det} \left({\bfa  M} \right)_{\Omega}  >0 
 \iff  d_{ \Omega} k_{l,m}^4 - \gamma_{ \Omega}\left( d_{ \Omega} f_u  + g_v \right)  k_{l,m}^2 + \gamma_{ \Omega}^2 (f_u g_v  -f_v g_u)>0,
\end{cases} \label{pd_cond1}
\end{equation}
or
\begin{equation}
\begin{cases}
\text{Tr} \left({\bfa  M}\right)_{\Omega} >0  \iff (d_{ \Omega} +1)k_{l,m}^2  -\gamma_{ \Omega} (f_u + g_v)>0,  \quad \text{and}  \\ \\
 \text{Det} \left({\bfa  M} \right)_{\Omega}  <0 \iff  d_{ \Omega} k_{l,m}^4 - \gamma_{ \Omega}\left( d_{ \Omega} f_u  + g_v \right)  k_{l,m}^2 + \gamma_{ \Omega}^2 (f_u g_v  -f_v g_u)<0.
\end{cases} \label{pd_cond2}
\end{equation}
Similarly, on the surface, $\text{Re} (\lambda_{l,m} (l(l+1))) >0$ for some $l(l+1)>0$ if and only the following conditions hold:
\begin{equation}
\begin{cases}
 \text{Tr} \left({\bfa  M}\right)_{\Gamma} <0  \iff (d_{ \Gamma} +1 )l(l+1) - \gamma_{ \Gamma} (f_r +g_s)  <0, \quad \text{and}  \\ \\
 \text{Det} \left({\bfa  M} \right)_{\Gamma}  >0 
 \iff d_{ \Gamma} l^2(l+1)^2 - \gamma_{ \Gamma} \left( d_{ \Gamma} f_r  + g_s \right)  l(l+1) + \gamma_{ \Gamma}^2 (f_r g_s - f_s g_r)>0, \label{pd_cond3}
\end{cases}
\end{equation}
or
\begin{equation}
\begin{cases}
 \text{Tr} \left({\bfa  M}\right)_{\Gamma} >0  \iff (d_{ \Gamma} +1 )l(l+1) - \gamma_{ \Gamma} (f_r +g_s)>0, \quad \text{and}  \\ \\
 \text{Det} \left({\bfa  M} \right)_{\Gamma}  <0 
  \iff d_{ \Gamma} l^2(l+1)^2 - \gamma_{ \Gamma} \left( d_{ \Gamma} f_r  + g_s \right)  l(l+1) + \gamma_{ \Gamma}^2 (f_r g_s - f_s g_r) <0. \label{pd_cond4}
\end{cases}
\end{equation}
We are in a position to state the following theorem:

\begin{theorem}
Assuming that
\begin{equation}
\text{Tr} \left({\bfa  J}_{\bfa  F}\right)_{\Omega}  = f_u + g_v <0 \quad \text{and} \quad  \text{Det} \left({\bfa  J}_{\bfa  F}\right)_{\Omega}  = f_u g_v - f_v g_u >0, \label{trace_omega} 
\end{equation}
then the necessary  conditions for $\text{Re} (\lambda_{l,m} (k_{l,m}^2))>0$ for some $k_{l,m}^2>0$ are given by
\begin{align}
d_{ \Omega}f_u + g_v  >0, \quad \text{and} \quad 
\left(d_{ \Omega}f_u + g_v \right)^2 - 4 d_{ \Omega} \left( f_u  g_v - f_v  g_u\right)>0. \label{turing_omega2}
\end{align}
Similarly, assuming that 
\begin{equation}
\text{Tr} \left({\bfa  J}_{\bfa  F}\right)_{\Gamma}   = f_r + g_s <0  \quad \text{and} \quad
\text{Det} \left({\bfa  J}_{\bfa  F}\right)_{\Gamma}  = f_r g_s - f_s g_r>0, \label{trace_gamma}
\end{equation}
 then the necessary conditions for $\text{Re} (\lambda_{l,m} (l(l+1))) >0$ for some $l(l+1)>0$  are given by
\begin{align}
d_{ \Gamma}f_r + g_s   >0, \quad \text{and} \quad 
\left(d_{ \Gamma}f_r + g_s \right)^2 - 4 d_{ \Gamma} \left( f_r  g_s - f_s  g_r \right)  >0. \label{turing_gamma2}
\end{align}
\end{theorem}
\begin{proof}
The proof is a direct  consequence of conditions \eqref{pd_cond1}- \eqref{pd_cond4}. Assuming conditions \eqref{trace_omega} and \eqref{trace_gamma} hold, then one of the conditions in \eqref{pd_cond1} and \eqref{pd_cond3} is violated, which implies that $\text{Re} (\lambda_{l,m} (k_{l,m}^2))<0$ for all $k_{l,m}^2>0$ and similarly $\text{Re} (\lambda_{l,m} (l(l+1))) <0$ for all $l(l+1)>0$. This entails that  the system can no  longer exhibit spatially inhomogeneous solutions.

The only two conditions left to hold true are \eqref{pd_cond2} and \eqref{pd_cond4}. This case corresponds to the classical standard two-component reaction-diffusion system which requires that (for details see for example \cite{murray2003})
\begin{align}
d_{ \Omega} f_u + g_v >0,  \quad \text{and} \quad 
\left(d_{ \Omega}f_u + g_v \right)^2 - 4 d_{ \Omega} \left( f_u  g_v - f_v  g_u\right)>0,
\end{align}
and similarly
\begin{align}
d_{ \Gamma}f_r + g_s  >0, \quad \text{and} \quad
\left(d_{ \Gamma}f_r + g_s \right)^2 - 4 d_{ \Gamma} \left( f_r  g_s - f_s  g_r \right)  >0.
\end{align}
This completes the proof.
\end{proof}
\begin{remark}
Assuming conditions \eqref{trace_omega} and \eqref{trace_gamma} both hold, then conditions \eqref{turing_omega2} and \eqref{turing_gamma2} imply that $d_{ \Omega} \ne 1$ and $d_{ \Gamma} \ne 1$.
\end{remark}
\begin{remark}\label{rem_both}
If condition \eqref{trace_omega} or  \eqref{trace_gamma} holds only, then either $d_{ \Omega} \ne 1$ or $d_{ \Gamma} \ne 1$ but not necessarily both.
\end{remark}

\begin{remark}
If conditions \eqref{trace_omega} and \eqref{trace_gamma} are both violated, then diffusion-driven instability can not occur.
\end{remark}

\begin{remark}
Similarly to classical reaction-diffusion systems, conditions \eqref{turing_omega2} and \eqref{turing_gamma2} imply the existence of critical diffusion coefficients in the bulk and on the surface whereby the uniform states lose stability.  In order for diffusion-driven instability to occur, the bulk and surface diffusion coefficients must be greater than the values of the critical diffusion coefficients. 
\end{remark}
Next we investigate under what assumptions on the reaction-kinetics do conditions \eqref{pd_cond1} and \eqref{pd_cond3} hold true.

\begin{itemize}
\item First let us consider the case  when
 \[ \text{Tr} \left({\bfa  J}_{\bfa  F}\right)_{\Omega}  = f_u + g_v >0 \;  \text{and} \;  \text{Det} \left({\bfa  J}_{\bfa  F}\right)_{\Omega}  = f_u g_v - f_v g_u >0, \]
and
\[ \text{Tr} \left({\bfa  J}_{\bfa  F}\right)_{\Gamma}   = f_r + g_s >0   \;  \text{and} \;
\text{Det} \left({\bfa  J}_{\bfa  F}\right)_{\Gamma}  = f_r g_s - f_s g_r>0.\]
Then
$  \text{Tr} \left({\bfa  J}_{\bfa  F}\right) =  \text{Tr} \left({\bfa  J}_{\bfa  F}\right)_{\Omega}  + \text{Tr} \left({\bfa  J}_{\bfa  F}\right)_{\Gamma}>0 $
which violates condition \eqref{ws_cond1}.

\item Similarly the case when
 \[ \text{Tr} \left({\bfa  J}_{\bfa  F}\right)_{\Omega}  = f_u + g_v >0 \;  \text{and} \;  \text{Det} \left({\bfa  J}_{\bfa  F}\right)_{\Omega}  = f_u g_v - f_v g_u <0, \]
and
\[ \text{Tr} \left({\bfa  J}_{\bfa  F}\right)_{\Gamma}   =f_r + g_s >0   \;  \text{and} \;
\text{Det} \left({\bfa  J}_{\bfa  F}\right)_{\Gamma}  =  f_r g_s - f_s g_r<0\]
violates  condition \eqref{ws_cond1}.
\item Let us consider the case when
 \[ \text{Tr} \left({\bfa  J}_{\bfa  F}\right)_{\Omega}  =f_u + g_v <0 \;  \text{and} \;  \text{Det} \left({\bfa  J}_{\bfa  F}\right)_{\Omega}  = f_u g_v - f_v g_u <0, \]
and
\[ \text{Tr} \left({\bfa  J}_{\bfa  F}\right)_{\Gamma}   = f_r + g_s<0   \;  \text{and} \;
\text{Det} \left({\bfa  J}_{\bfa  F}\right)_{\Gamma}  =  f_r g_s - f_s g_r<0.\]
Then it follows that condition \eqref{ws_cond5} given by
\begin{align}
& \Big[\text{Tr} \left({\bfa  J}_{\bfa  F}\right)_{\Gamma} \text{Tr} \left({\bfa  J}_{\bfa  F}\right)  - 2  \text{Det} \left({\bfa  J}_{\bfa  F}\right)_{\Omega}  \Big]  \text{Tr} \left({\bfa  J}_{\bfa  F}\right)_{\Omega}  + \Big[\text{Tr} \left({\bfa  J}_{\bfa  F}\right)_{\Omega} \text{Tr} \left({\bfa  J}_{\bfa  F}\right)  \Big.  \notag \\ & \Big. \hspace{18em}- 2  \text{Det} \left({\bfa  J}_{\bfa  F}\right)_{\Gamma}  \Big]  \text{Tr} \left({\bfa  J}_{\bfa  F}\right)_{\Gamma}   <0,  \notag
\end{align}
is violated.
\item Next we consider the case when
 \[ \text{Tr} \left({\bfa  J}_{\bfa  F}\right)_{\Omega}  = f_u + g_v <0 \;  \text{and} \;  \text{Det} \left({\bfa  J}_{\bfa  F}\right)_{\Omega}  = f_u g_v - f_v g_u <0, \]
and
\[ \text{Tr} \left({\bfa  J}_{\bfa  F}\right)_{\Gamma}   =f_r + g_s>0   \;  \text{and} \;
\text{Det} \left({\bfa  J}_{\bfa  F}\right)_{\Gamma}  =  f_r g_s - f_s g_r<0.\]
It follows then that none of the conditions \eqref{ws_cond1}-\eqref{ws_cond6} are violated.  However, condition \eqref{pd_cond1} does not hold. 
\item Similarly the case when
 \[ \text{Tr} \left({\bfa  J}_{\bfa  F}\right)_{\Omega}  =f_u + g_v >0 \;  \text{and} \;  \text{Det} \left({\bfa  J}_{\bfa  F}\right)_{\Omega}  = f_u g_v - f_v g_u<0, \]
and
\[ \text{Tr} \left({\bfa  J}_{\bfa  F}\right)_{\Gamma}   = f_r + g_s <0   \;  \text{and} \;
\text{Det} \left({\bfa  J}_{\bfa  F}\right)_{\Gamma}  =  f_r g_s - f_s g_r<0.\]
This implies that that none of the conditions \eqref{ws_cond1}-\eqref{ws_cond6} are violated, while condition \eqref{pd_cond3} fails not hold.  
\item Finally, the cases when
 \begin{equation}
 \begin{cases}
  \text{Tr} \left({\bfa  J}_{\bfa  F}\right)_{\Omega}  = f_u + g_v >0 \;  \text{and} \;  \text{Det} \left({\bfa  J}_{\bfa  F}\right)_{\Omega}  = f_u g_v - f_v g_u>0, \\
 \text{Tr} \left({\bfa  J}_{\bfa  F}\right)_{\Gamma}   = f_r + g_s <0   \;  \text{and} \;
\text{Det} \left({\bfa  J}_{\bfa  F}\right)_{\Gamma}  =  f_r g_s - f_s g_r>0,
\end{cases}
\end{equation}
and   
\begin{equation}
\begin{cases}
\text{Tr} \left({\bfa  J}_{\bfa  F}\right)_{\Omega}  = f_u + g_v <0 \;  \text{and} \;  \text{Det} \left({\bfa  J}_{\bfa  F}\right)_{\Omega}  =f_u g_v - f_v g_u >0, \\
\text{Tr} \left({\bfa  J}_{\bfa  F}\right)_{\Gamma}   = f_r + g_s >0   \;  \text{and} \; \text{Det} \left({\bfa  J}_{\bfa  F}\right)_{\Gamma}  =  f_r g_s - f_s g_r>0,
\end{cases}
\end{equation}  
result in {\it Remark} \ref{rem_both} above.
\end{itemize}
The above cases clearly eliminate conditions \eqref{pd_cond1} and \eqref{pd_cond3} as necessary for uniform steady state to be driven unstable in the presence of diffusion. We are now in a position to state our main result.
\begin{theorem}[Necessary conditions for diffusion-driven instability]\label{ns_cond_ad2}
The necessary conditions for diffusion-driven instability for  the coupled system of  BSRDEs
 \eqref{bs_rdes_model0} - \eqref{hv_eqn2}
 are given by
\begin{align}
& \text{Tr} \left({\bfa  J}_{\bfa  F}\right)   <0, \label{nsuf_cond1} \\
& \text{Det} \left({\bfa  J}_{\bfa  F}\right)_{\Omega} + \text{Det} \left({\bfa  J}_{\bfa  F}\right)_{\Gamma} + \text{Tr} \left({\bfa  J}_{\bfa  F}\right)_{\Omega}\text{Tr} \left({\bfa  J}_{\bfa  F}\right)_{\Gamma}  >0, \ \label{nsuf_cond2}\\
& \text{Det} \left({\bfa  J}_{\bfa  F}\right)_{\Omega}  \text{Tr} \left({\bfa  J}_{\bfa  F}\right)_{\Gamma} + \text{Det} \left({\bfa  J}_{\bfa  F}\right)_{\Gamma} \text{Tr} \left({\bfa  J}_{\bfa  F}\right)_{\Omega}  <0,  \label{nsuf_cond3} \\
& \text{Det} \left({\bfa  J}_{\bfa  F}\right)_{\Omega}  \text{Det} \left({\bfa  J}_{\bfa  F}\right)_{\Gamma}  >0,  \label{nsuf_cond4} \\
& \Big[\text{Tr} \left({\bfa  J}_{\bfa  F}\right)_{\Gamma} \text{Tr} \left({\bfa  J}_{\bfa  F}\right)  - 2  \text{Det} \left({\bfa  J}_{\bfa  F}\right)_{\Omega}  \Big]  \text{Tr} \left({\bfa  J}_{\bfa  F}\right)_{\Omega}  + \Big[\text{Tr} \left({\bfa  J}_{\bfa  F}\right)_{\Omega} \text{Tr} \left({\bfa  J}_{\bfa  F}\right)  \Big.  \notag \\ & \Big. \hspace{20em}- 2  \text{Det} \left({\bfa  J}_{\bfa  F}\right)_{\Gamma}  \Big]  \text{Tr} \left({\bfa  J}_{\bfa  F}\right)_{\Gamma}    >0,   \label{nsuf_cond5} \\
& \left[ \big(\text{Det} \left({\bfa  J}_{\bfa  F}\right)_{\Omega}  +  \text{Det} \left({\bfa  J}_{\bfa  F}\right)_{\Gamma} \big)^2 - \big( \text{Det} \left({\bfa  J}_{\bfa  F}\right)_{\Omega} \text{Tr} \left({\bfa  J}_{\bfa  F}\right)_{\Gamma}  \right. \notag \\ & \left. \hspace{8em} + \text{Det} \left({\bfa  J}_{\bfa  F}\right)_{\Gamma} \text{Tr} \left({\bfa  J}_{\bfa  F}\right)_{\Omega} \big) \text{Tr} \left({\bfa  J}_{\bfa  F} \right)  \right] \text{Tr} \left({\bfa  J}_{\bfa  F}\right)_{\Omega} \text{Tr} \left({\bfa  J}_{\bfa  F}\right)_{\Gamma}>0,  \label{nsuf_cond6}
\end{align}
and
\begin{equation}
d_{ \Omega}f_u + g_v >0,  \quad \text{and} \quad 
\left(d_{ \Omega}f_u + g_v \right)^2 - 4 d_{ \Omega} \text{Det} \left({\bfa  J}_{\bfa  F}\right)_{\Omega} >0,
\label{nsturing_omega}
\end{equation}
or/and
\begin{equation}
d_{ \Gamma}f_r + g_s  >0, \quad \text{and} \quad 
\left(d_{ \Gamma}f_r + g_s \right)^2 - 4 d_{ \Gamma}\text{Det} \left({\bfa  J}_{\bfa  F}\right)_{\Gamma}  >0.
 \label{nsturing_gamma}
\end{equation}
\end{theorem}

\subsubsection{Theoretical predictions}\label{predictions}
From the analytical results we state the following theoretical predictions to be validated through the use of numerical simulations. 
\begin{enumerate}
\item The bulk and surface diffusion coefficients $d_{\Omega}$ and $d_{\Gamma}$ must be greater than one in order for diffusion-driven instability to occur. Taking $d_{\Omega}=d_{\Gamma}=1$ results in a contradiction between conditions \eqref{nsuf_cond1}, \eqref{nsturing_omega} and \eqref{nsturing_gamma}. As a result, the BSRDEs does not give rise to the formation of spatial structure. For this case, the uniform steady state is the only stable solution for the coupled system of BSRDEs  \eqref{bs_rdes_model0} - \eqref{hv_eqn2}.
\item The above imply that taking $d_{\Omega} >1$ and $d_{\Gamma}=1$, the bulk-reaction-diffusion system has the potential to induce patterning in the bulk for appropriate diffusion-driven instability parameter values while the surface-reaction-diffusion system is not able to generate patterns. Here all the conditions \eqref{nsuf_cond1}-\eqref{nsturing_omega} hold except \eqref{nsturing_gamma}.
\item Alternatively taking  $d_{\Omega} =1$ and $d_{\Gamma}>1$, the bulk-reaction-diffusion system fails to induce patterning in the bulk while the surface-reaction-diffusion system has the potential to induce patterning  on the surface. Similarly,  all the conditions \eqref{nsuf_cond1}-\eqref{nsturing_gamma} hold except \eqref{nsturing_omega}. 
\item On the other hand, taking $d_{\Omega} >1$ and $d_{\Gamma}>1$ appropriately, then the coupled system of BSRDEs exhibits patterning both in the bulk and on the surface. All the conditions  \eqref{nsuf_cond1}-\eqref{nsturing_gamma} hold both in the bulk and on the surface.
\end{enumerate}
\section{Numerical simulations of the coupled system of bulk-surface reaction-diffusion equations (BSRDEs)}\label{sec:numerics}
In this section we present bulk-surface finite element numerical solutions corresponding to the coupled system of bulk-surface reaction diffusion equations (BSRDEs) \eqref{bs_rdes_model0}-\eqref{schnakv}.  Here we omit the details of the bulk-surface finite element method as these are given elsewhere (see \cite{madzvamuse2014b} for details). Our method is inspired by the work of Elliott and Ranner \cite{elliott2012}. We use the bulk-surface finite element method  to discretise in space with piecewise bilinear elements and an implicit second order fractional-step $\theta$-scheme to discretise in time using the Newton's method for the lineraisation \cite{madzvamuse2014a,madzvamuse2014b}. For details on the convergence and stability of the fully implicit time-stepping fractional-step $\theta$-scheme, the reader is referred to Madzvamuse {\it et al.} \cite{madzvamuse2014a,madzvamuse2014b}. In all our numerical experiments, we fix the kinetic model parameter values $a=0.1$ and $b=0.9$ since these satisfy the Turing diffusion-driven instability conditions \eqref{nsuf_cond1}-\eqref{nsturing_gamma}.  For these model parameter values the equilibrium values are $(u^*,v^*,r^*,s^*)=(1,0.9,1,0.9)$. Initial conditions are prescribed as small random perturbations around the equilibrium values. For illustrative purposes let us take the parameter values describing the boundary conditions as shown in Table \ref{para_boundary}; these are selected such that they satisfy the compatibility condition \eqref{c1}.
\begin{table}[ht!]
   \centering
   \begin{tabular}{|c|c|c|c|c|c|c|c|c|c|c|c|c|}\hline
   $a$ & $b $  & $\gamma_{\Omega}$  & $\gamma_{\Gamma}$& $\alpha_1$ & $\alpha_2$ &$\beta_1$ &$\beta_2$ &$\kappa_1$ &$\kappa_2$  \\\hline
    0.1 & 0.9   & 500 & 500&  $\frac{5}{12}$ & 5 & $\frac{5}{12}$ & $0$ & $0$ & 5    \\\hline
       \end{tabular}
    \caption{Model parameter values for  the coupled system of BSRDEs \eqref{bs_rdes_model0} - \eqref{hv_eqn2}.}
         \label{para_boundary}
\end{table}

\subsubsection{A note on the bulk-surface triangulation}
We briefly outline how the bulk-surface triangulation is generated. For further specific details please see reference \cite{madzvamuse2014b}. 
Let $\Omega_h$ be a triangulation of the bulk geometry $\Omega$ with vertices ${\bfa x}_i$, $i=1,...,N_h$, where $N_h$ is the number of vertices in the triangulation. From $\Omega_h$ we then construct $\Gamma_h$ to be the triangulation of the surface geometry $\Gamma$ by defining $\Gamma_h = \Omega_h|_{\partial \Omega_h}$, i.e. the vertices of $\Gamma_h$ are the same as those lying on the surface of $\Omega_h$. In particular, then, we have $\partial\Omega_h=\Gamma_h$. An example mesh is shown in Figure \ref{mesh_example}.  The bulk triangulation induces the surface triangulation as illustrated. 

\begin{figure}
 \centering 
 \includegraphics[keepaspectratio=true,width=.9\textwidth]{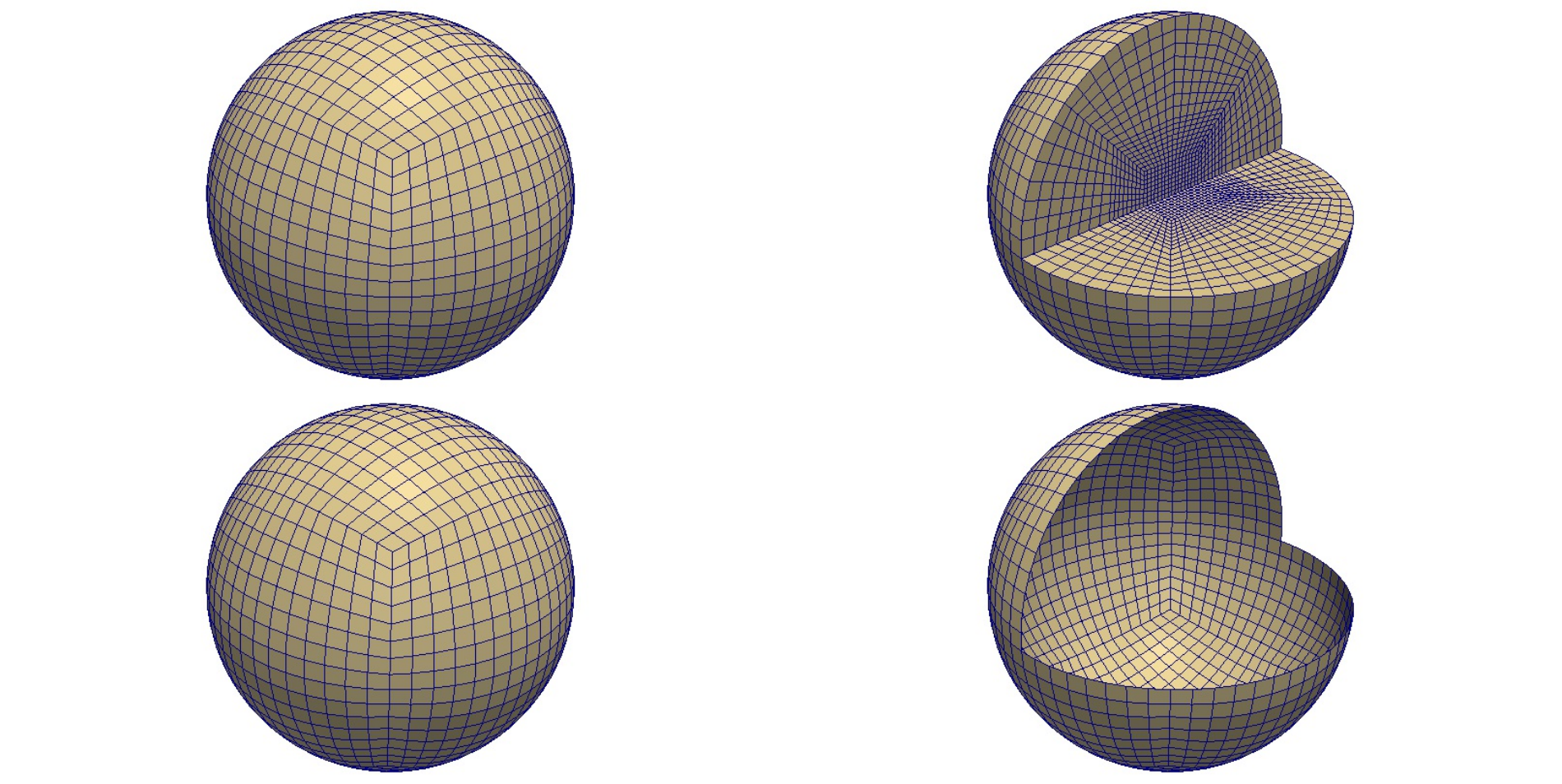}
\caption{Example meshes for the bulk (top) and surface system (bottom). Part of the domain has been cut away and shown on the right to reveal some internal mesh structure.}
\label{mesh_example}
\end{figure}


\subsection{Numerical experiments}
In this section we will only present four cases to validate our theoretical predictions outlined in Section \ref{predictions}. In most of  our simulations parameter values are fixed as shown in Table \ref{para_boundary}, except for $d_{ \Omega}$ and  $d_{ \Gamma}$ whose values are varied to demonstrate the patterning mechanism of the coupled system of BSRDEs. We only present patterns corresponding to the chemical species $u$ and $r$ in the bulk and on the surface respectively. Those corresponding to $v$ and $s$  are 180 degrees out of phase to those of $u$ and  $r$ and are therefore omitted. It must be noted however that this is not always the case in general, Robin-type boundary conditions may alter the structure of the solution profiles depending on the model parameter values and the coupling compatibility boundary parameters.

\subsubsection{\it Simulations of the coupled system of BSRDEs with $(d_{ \Omega},d_{ \Gamma})=(1,1)$}
\begin{figure}[htb]
\begin{center}
     \subfigure{
\includegraphics[keepaspectratio=true,width=.48\textwidth]{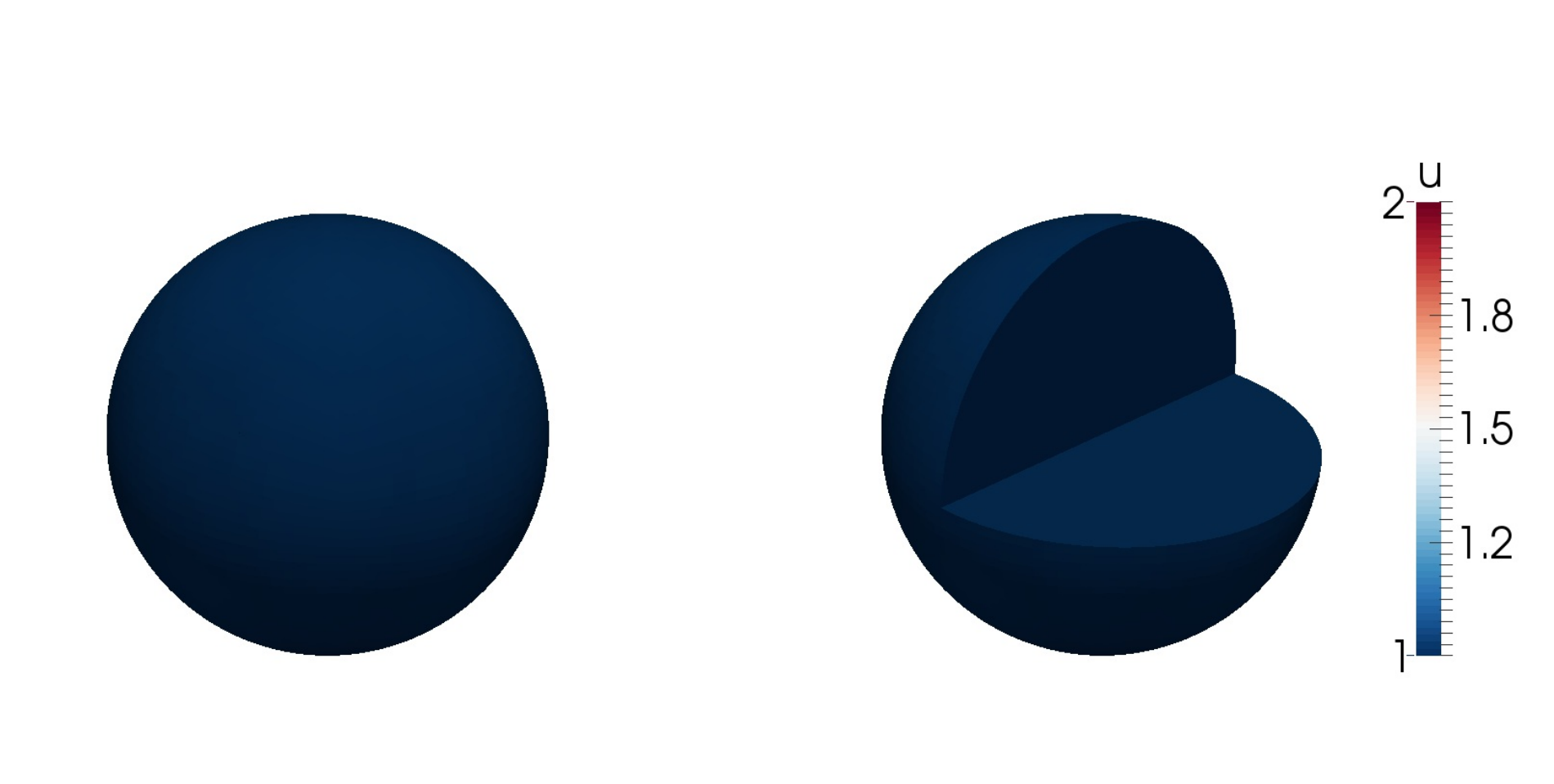}}
	\subfigure{
\includegraphics[keepaspectratio=true,width=.48\textwidth]{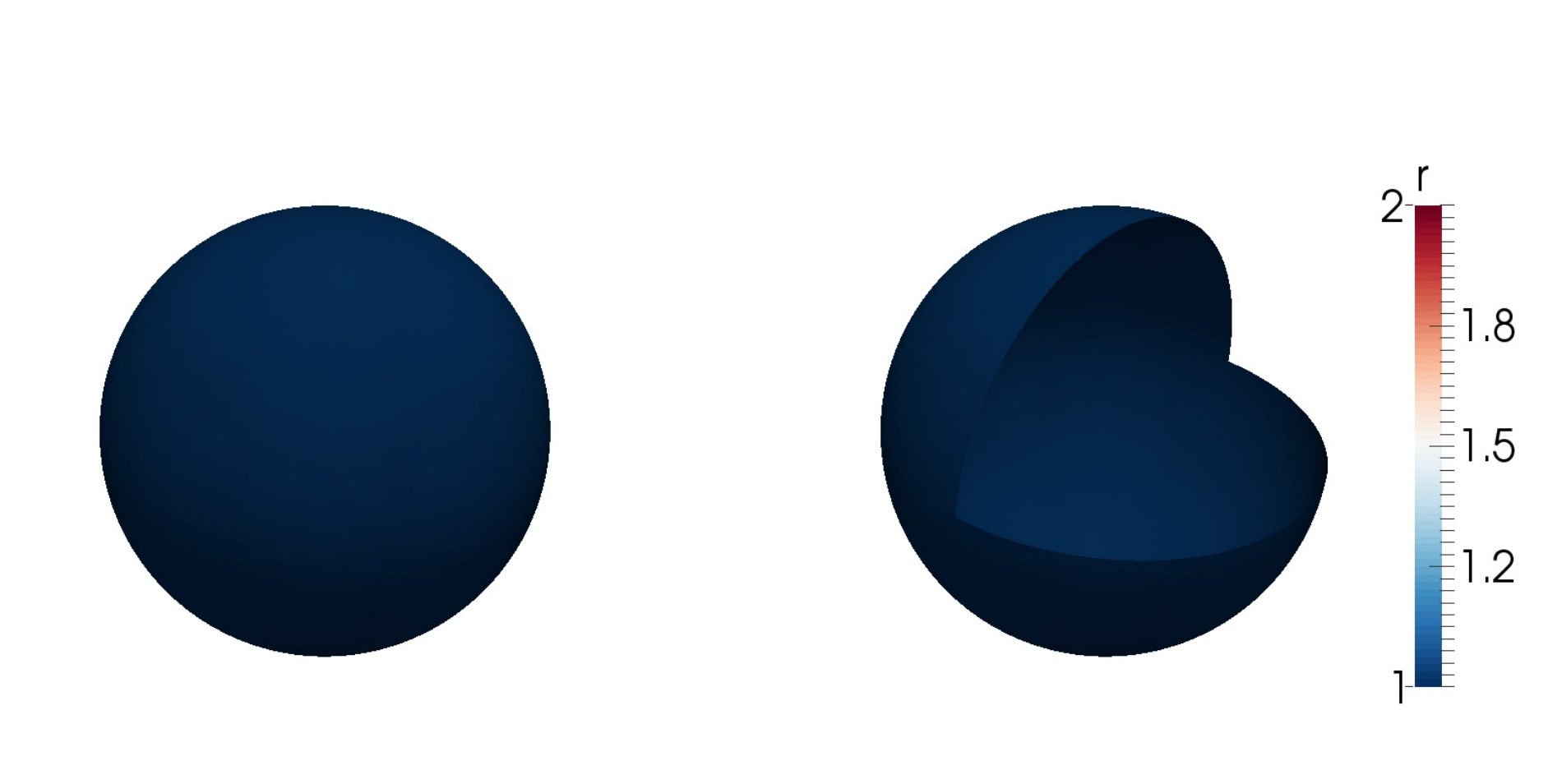}}
\end{center}
\caption{Numerical solutions corresponding to the coupled system of BSRDEs \eqref{bs_rdes_model0}-\eqref{schnakv}  with $d_{ \Omega}=1$ in the bulk and $d_{ \Gamma}=1$ on the surface. The uniform steady state solutions are converged to and no patterns form. Columns 1 and 2: solutions in the bulk representing $u$. Columns 3 and 4: solutions on the surface representing $r$. Second and fourth columns represent cross sections of the bulk and the surface respectively (Colour figure online). }
\label{fig:bd_1_sd_1_g500}
\end{figure}
The bulk-surface finite element numerical simulations of the coupled system of BSRDEs with with $d_{ \Omega}=1$ in the bulk, $d_{ \Gamma}=1$ on the surface are shown in Figure \ref{fig:bd_1_sd_1_g500}. We observe that no patterns form in complete agreement with theoretical predictions. Similarly to classical reaction-diffusion systems, diffusion coefficients must be greater than one. In particular, the diffusion coefficients must be greater than their corresponding respective critical diffusion coefficients in the bulk and on the surface. An example is shown next.  

\subsubsection{\it Simulations of the coupled system of BSRDEs with $(d_{ \Omega},d_{ \Gamma})=(1,20)$}
\begin{figure}[htb]
\begin{center}
     \subfigure{
\includegraphics[keepaspectratio=true,width=.48\textwidth]{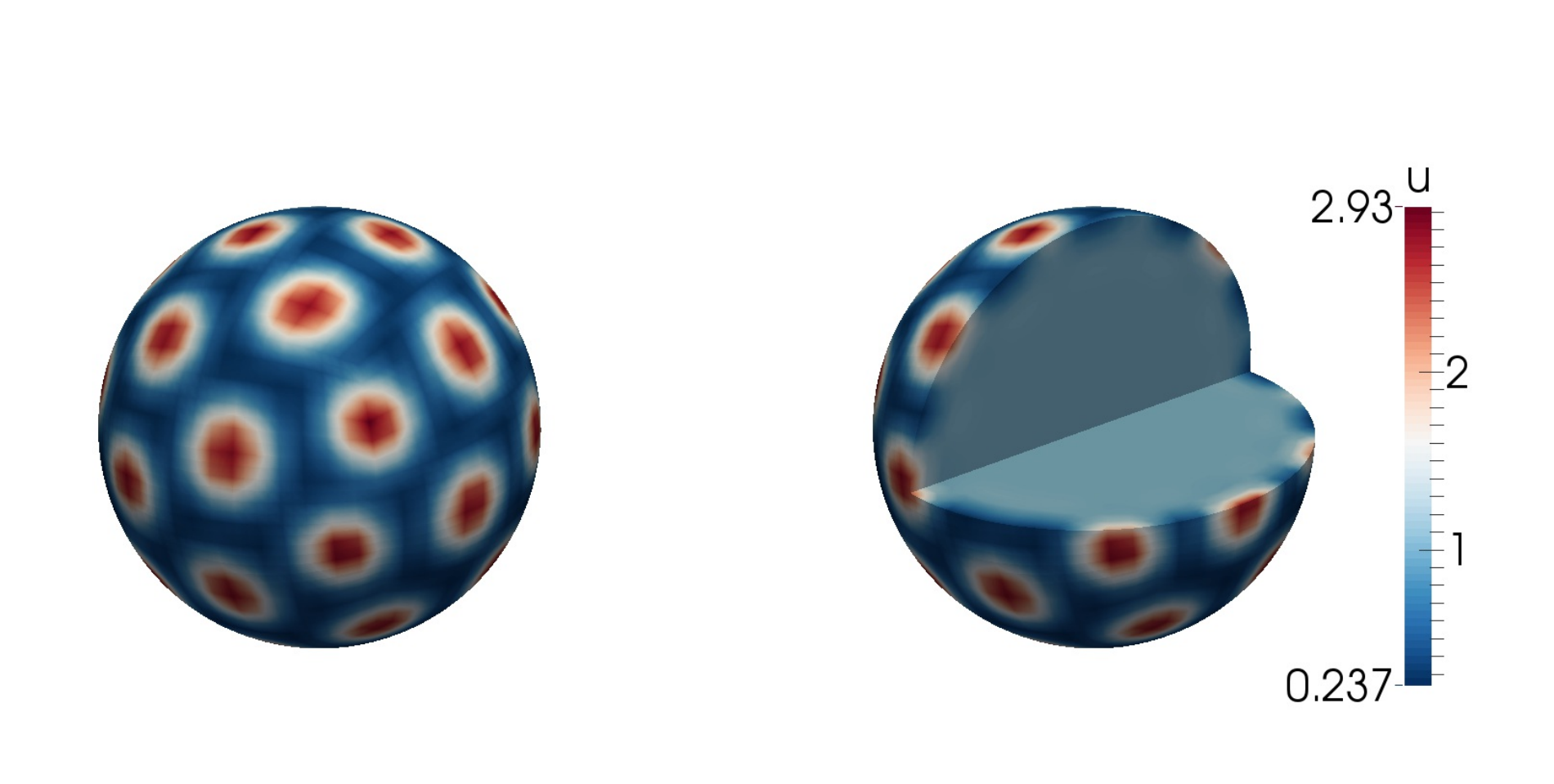}}
	\subfigure{
\includegraphics[keepaspectratio=true,width=.48\textwidth]{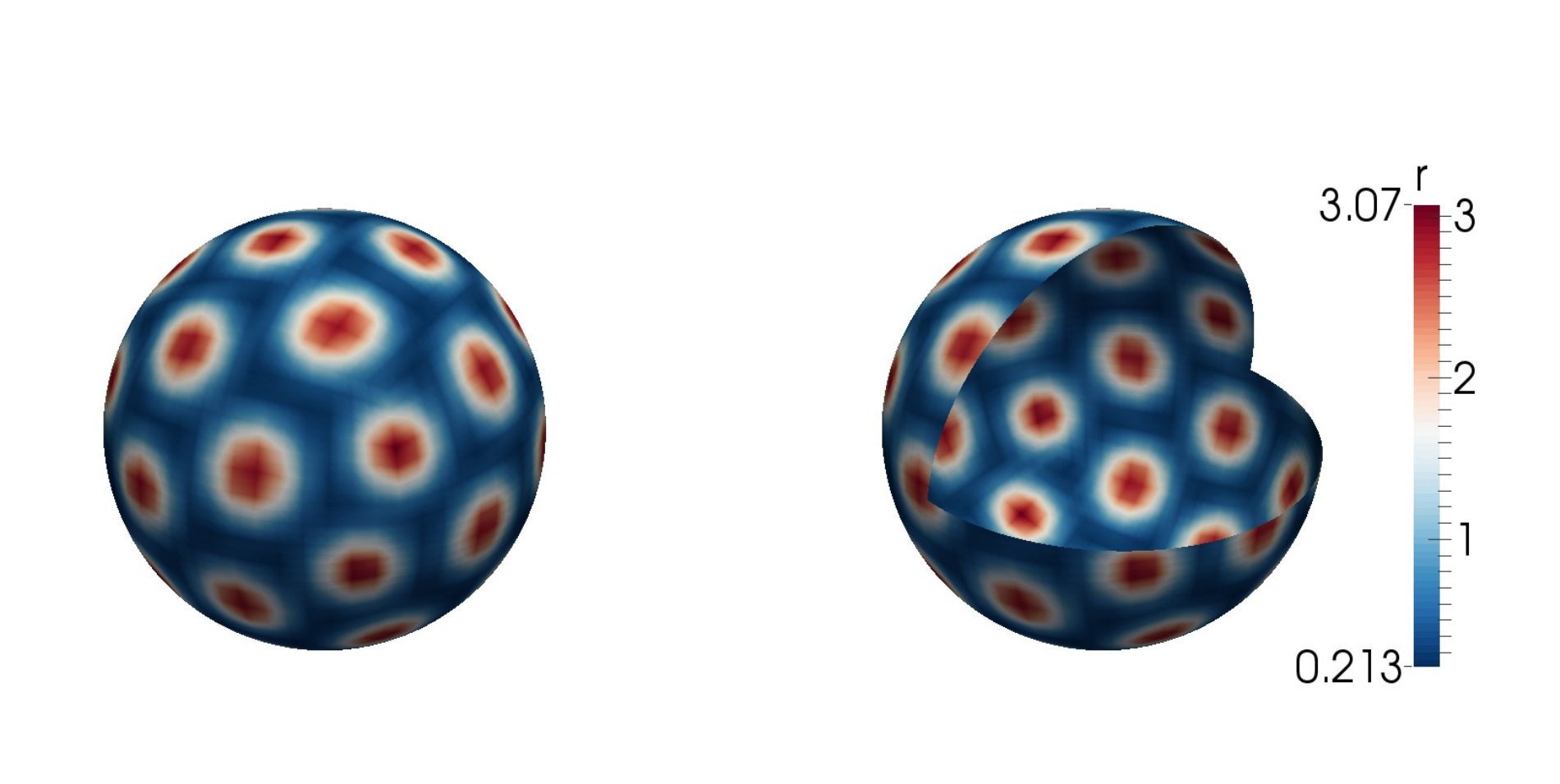}}
\end{center}
\caption{Numerical solutions corresponding to the coupled system of BSRDEs \eqref{bs_rdes_model0}-\eqref{schnakv}  with $d_{ \Omega}=1$ in the bulk and $d_{ \Gamma}=20$ on the surface. Columns 1 and 2: solutions in the bulk representing $u$. Columns 3 and 4: solutions on the surface representing $r$. Second and fourth columns represent cross sections of the bulk and the surface respectively (Colour figure online). Spot patterns form on the surface while  small balls form in the vicinity of the surface inside the bulk.}
\label{fig:bd_1_sd_20_g500}
\end{figure}
For illustrative purposes, let us take  $d_{ \Omega}=1$ in the bulk, $d_{ \Gamma}=20 > d_{ \Gamma}^{\it crit} = 8.5$ on the surface. Figure \ref{fig:bd_1_sd_20_g500} illustrate pattern formation on the surface as well as within a small region in the vicinity of the surface membrane. Spots are observed to form on the surface, while in the bulk, small balls form inside. Far away from the surface, no patterns form since the necessary conditions for diffusion-driven instability are not fulfilled in the bulk. These results confirm our theoretical predictions. We note that this particular example describes realistically pattern formation in biological systems. We expect skin patterning to manifest in the epidermis layer as well as on the surface. 

\subsubsection{\it Simulations of the coupled system of BSRDEs with $(d_{ \Omega},d_{ \Gamma})=(20,1)$}
\begin{figure}[htb]
\begin{center}
     \subfigure{
\includegraphics[keepaspectratio=true,width=.48\textwidth]{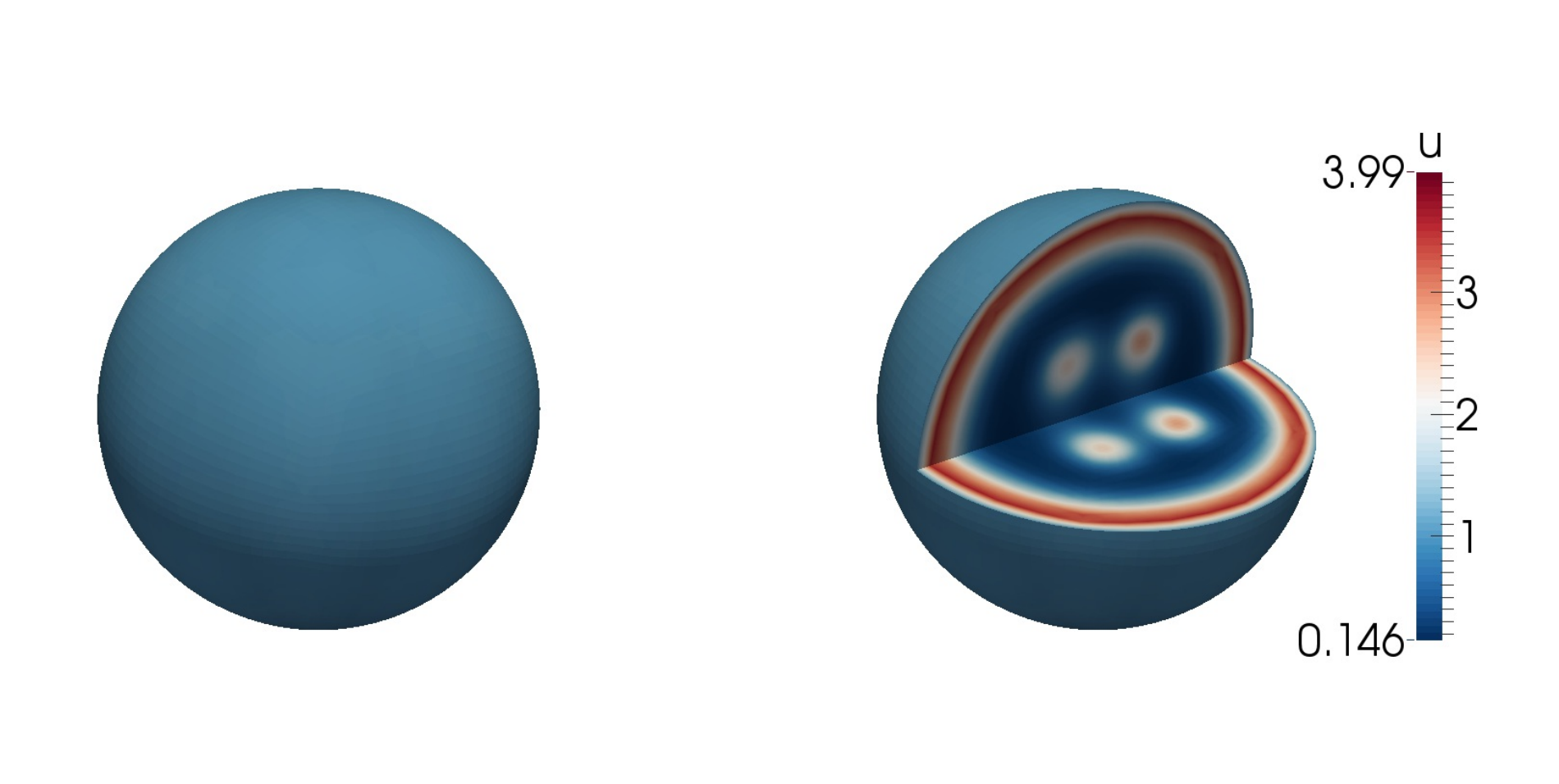}}
	\subfigure{
\includegraphics[keepaspectratio=true,width=.48\textwidth]{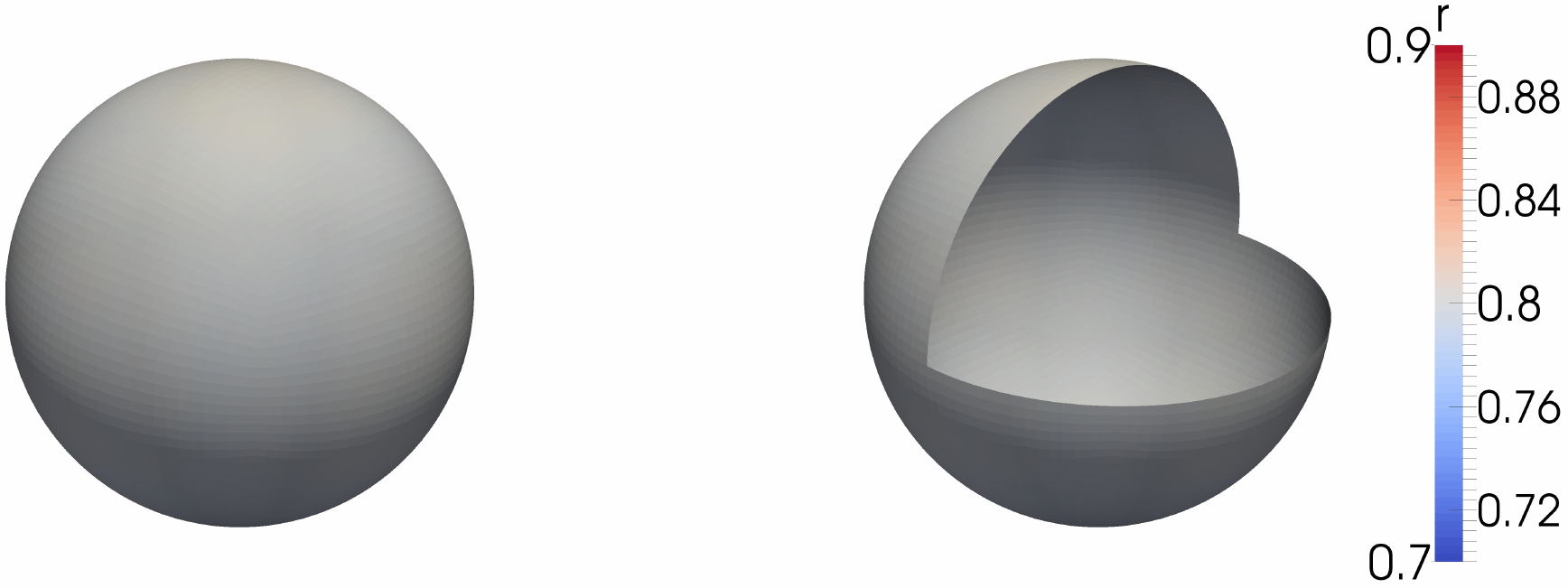}}
\end{center}
\caption{Numerical solutions corresponding to the coupled system of BSRDEs \eqref{bs_rdes_model0}-\eqref{schnakv}  with $d_{ \Omega}=20$ in the bulk and $d_{ \Gamma}=1$ on the surface. Columns 1 and 2: solutions in the bulk representing $u$. Columns 3 and 4: solutions on the surface representing $r$. Second and fourth columns represent cross sections of the bulk and the surface respectively (Colour figure online). Spectacular patterning occurs in the bulk exhibiting spots, stripes and circular patterns. The surface dynamics produce uniform patterning.}
\label{fig:bd_20_sd_1_g500}
\end{figure}
%

To generate patterns in the bulk we take $d_{ \Omega}=20> d_{ \Gamma}^{\it crit} = 8.5$ and  $d_{ \Gamma}=1$ on the surface. Figure \ref{fig:bd_20_sd_1_g500} exhibits stripe, circular and spot patterns in the bulk as illustrated by the cross-sections.  On the surface, uniform patterns occur consistent with theoretical predictions. Although the patterns for the $u$ species (columns one and two) appear uniform on the surface this is simply due to the colour scale, with the amplitude of the patterns in the bulk larger than those on the surface. This difference in the amplitude of the pattern of the bulk solution in the bulk and on the surface is due to the Robin type boundary conditions. Unlike zero-flux (also known as homogeneous Neumann) boundary conditions for standard reaction-diffusion systems which imply that no species enter or leave the domain, here, there is deposition or removal of chemical species through the flux on the surface, resulting in differences in amplitude between the bulk and surface solutions.

\subsubsection{\it Simulations of the coupled system of BSRDEs with $(d_{ \Omega},d_{ \Gamma})=(20,20)$}
\begin{figure}[htb]
\begin{center}
     \subfigure{
\includegraphics[keepaspectratio=true,width=.48\textwidth]{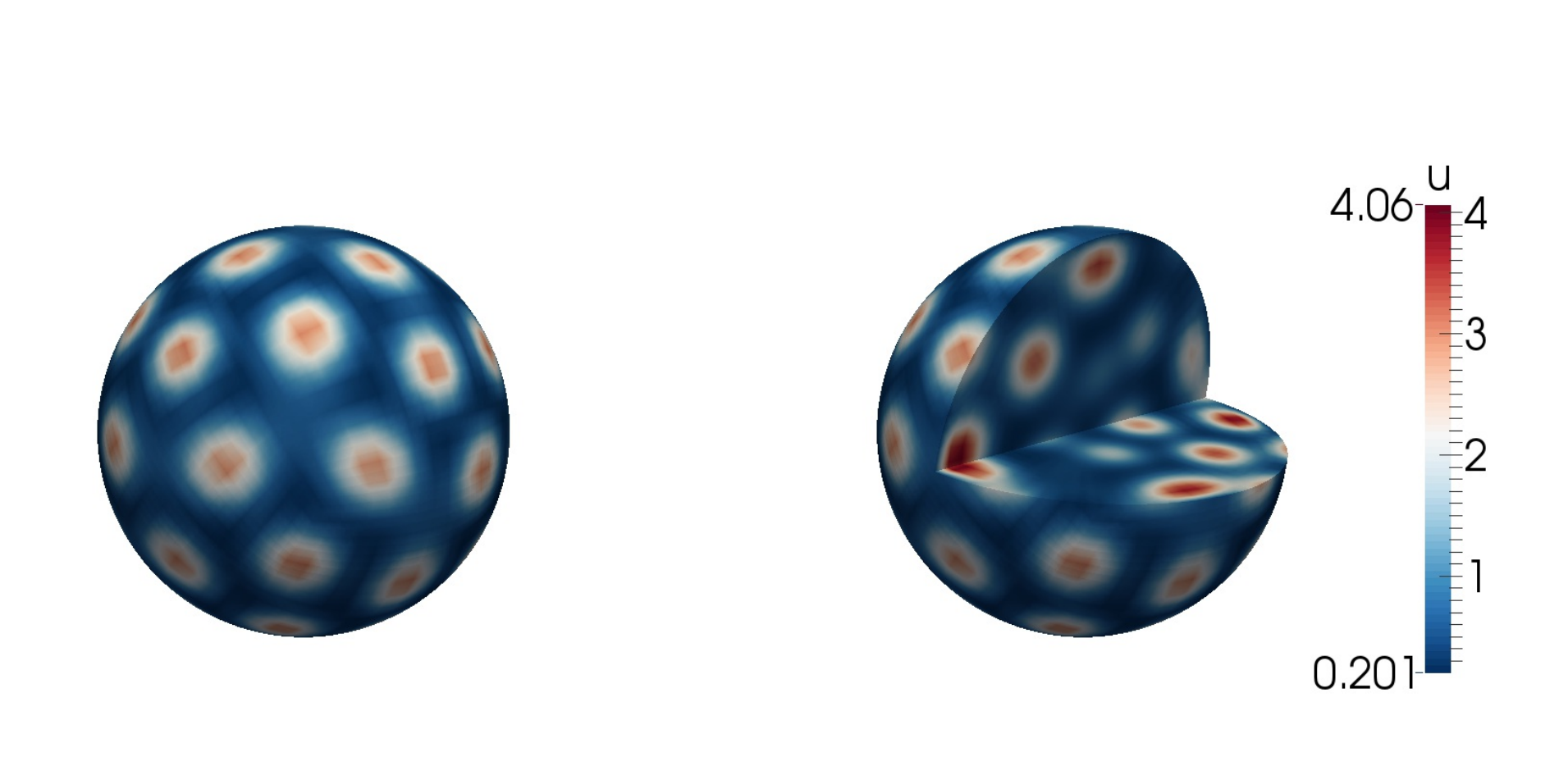}}
	\subfigure{
\includegraphics[keepaspectratio=true,width=.48\textwidth]{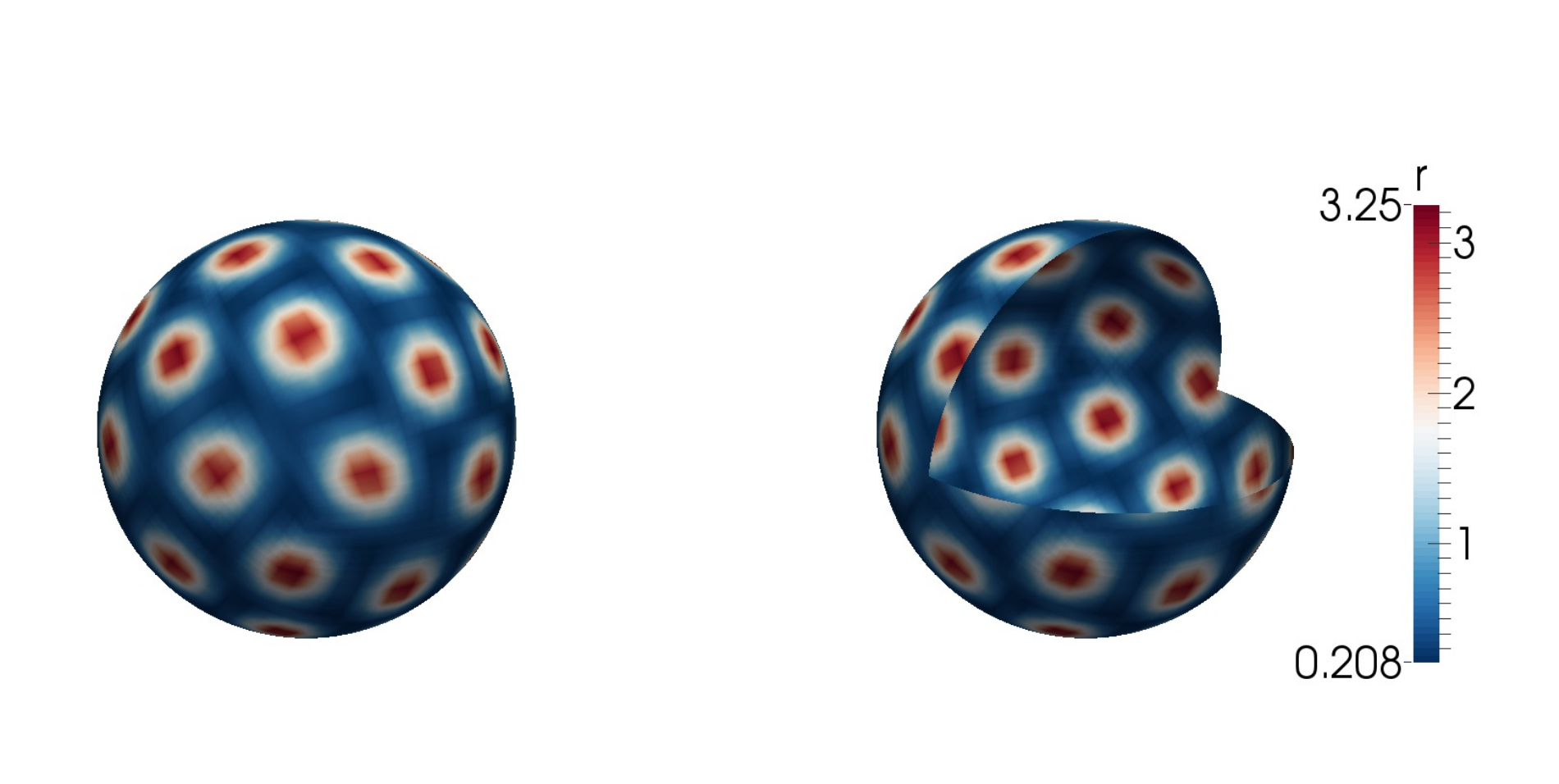}}
\end{center}
\caption{Numerical solutions corresponding to the coupled system of BSRDEs \eqref{bs_rdes_model0}-\eqref{schnakv}  with $d_{ \Omega}=20$ in the bulk and $d_{ \Gamma}=20$ on the surface. Columns 1 and 2: solutions in the bulk representing $u$. Columns 3 and 4: solutions on the surface representing $r$. Second and fourth columns represent cross sections of the bulk and the surface respectively (Colour figure online). We observe spot pattern formation both in the bulk and on the surface. }
\label{fig:bd_20_sd_20_g500}
\end{figure}
In this example, we illustrate how both bulk and surface dynamics induce patterning by taking $d_{ \Omega}=20$ in the bulk, $d_{ \Gamma}=20$ on the surface. Figure \ref{fig:bd_20_sd_20_g500} shows pattern formation in the bulk and on the surface. In the bulk we observe the formation of balls (which can be seen as spots through cross-sections) and these translate to spots on the surface. The surface dynamics themselves induce spot pattern formation. 

\section{Conclusion, discussion and future research challenges} \label{sec:conclusion}
We have presented a coupled system of bulk-surface reaction-diffusion equations whereby the bulk and surface reaction-diffusion systems are coupled through Robin-type boundary conditions. Nonlinear reaction-kinetics are considered in the bulk and on the surface and for illustrative purposes, the activator-depleted model was selected since it has a unique positive steady state. By using linear stability theory close to the bifurcation point, we state and prove a generalisation of the necessary conditions for Turing diffusion-driven instability for the coupled system of BSRDEs. Our most revealing result is that the bulk reaction-diffusion system has the capability of inducing patterning (under appropriate model and compatibility parameter values) for the surface reaction-diffusion model. On the other hand, the surface reaction-diffusion is not capable of inducing patterning everywhere in the bulk; patterns can only be induced in regions close to the surface membrane. For skin pattern formation, this example is consistent with the observation that patterns will form on the surface as well as within the epidermis layer close to the surface. We do not expect patterning to form everywhere in the body of the animals. 

Our studies reveal the following observations and research questions still to addressed:
\begin{itemize}
\item Our numerical experiments reveal that the Robin-type boundary conditions seem to introduce a boundary layer coupling the bulk and surface dynamics. However, these boundary conditions do not appear explicitly in the conditions for diffusion-driven instability and this makes it difficult to theoretically analyse their role and implications to pattern formation. Further studies are required to understand the role of these boundary conditions as well as the size of the boundary layer.
\item The compatibility condition \eqref{c1} implies that the uniform steady state in the bulk is identical to the uniform state on the surface. We are currently studying the implications of relaxing the compatibility condition.  
\item Finally, in this manuscript, we have not carried out detailed parameter search and estimation to deduce the necessary and sufficient conditions for pattern generation as well isolating excitable wavenumbers in the bulk and on the surface. Such studies might reveal more interesting properties of the coupled bulk-surface model and this forms part of our current studies. 
\end{itemize}

We have presented a framework that couples bulk dynamics ($3D$) to surface dynamics ($2D$) with the potential of numerous applications in cell motility, developmental biology, tissue engineering and regenerative medicine and biopharmaceutical where reaction-diffusion type models are routinely used \cite{chechkin2012,elliott2013,medvedev2013,nisbet2009,novak2007,ratz2012,ratz2013,venkpre}. 

We have restricted our studies to stationary volumes. In most cases, biological surfaces are known to evolve continuously with time. This introduces extra complexities to the modelling, analysis and simulation of coupled systems of bulk-surface reaction-diffusion equations. In order to consider evolving bulk-surface partial differential equations, evolution laws (geometric) should be formulated describing how the bulk and surface evolve. Here, it is important to consider specific experimental settings that allow for detailed knowledge of properties (biomechanical) and processes (biochemical) involved in the bulk-surface evolution. Such a framework will allows us to  study 3D cell migration in the area of cell motility \cite{elliott2013,george2013,madzvamuse2013,neilson2011}. In future studies, we propose to develop a 3D integrative model that couples bulk and surface dynamics during growth development or movement. 

\paragraph{ \textnormal{\textbf{Data accessibility.} This manuscript does not contain primary data and as a result has no supporting material associated with the results presented. }}

\paragraph{ \textnormal{\textbf{Competing interests.} The authors have no competing interest.}}

\paragraph{ \textnormal{\textbf{Authors' contributions.} The authors contributed equally to this work.}} 

\paragraph{ \textnormal{\textbf{Acknowledgements.}
The authors want to thank anonymous reviewers for their constructive comments.}}

\paragraph{ \textnormal{\textbf{Funding.}
This work (AM and CV) is supported by the Engineering and Physical Sciences Research Council grant: (EP/J016780/1).  AM and CV acknowledge support from the Leverhulme Trust  Research Project Grant (RPG-2014-149).  AHC was supported partly by the University of Sussex and partly by the Medical Research Council.}}

\end{document}